\documentclass{article}

\usepackage[T1]{fontenc}
\usepackage[english]{babel}

\usepackage{lipsum}
\usepackage{amsfonts}
\usepackage{graphicx}
\usepackage{epstopdf}
\usepackage{algorithmic}
\ifpdf
  \DeclareGraphicsExtensions{.eps,.pdf,.png,.jpg}
\else
  \DeclareGraphicsExtensions{.eps}
\fi

\usepackage{amsmath}
\usepackage{amssymb}
\usepackage{pgfplotstable}
\usepackage{booktabs}

\newfont{\NUMBERS}{msbm9 scaled\magstep1}

\newcommand{\REAL}{ \mathbb{R}}

\usepackage{ifthen}

\newcommand{\REALP}{\REAL_{+}}

\newcommand{\Gap}{\mbox{DGap}}

\newcommand{\Of}[1]{\left[#1\right]}
\newcommand{\Vect}[2][]
{
  \ifthenelse{\equal{#1}{}}
  {\boldsymbol{#2}}
  {#2_{#1}}
}
\newcommand{\Matr}[2][]
{
  \ifthenelse{\equal{#1}{}}
  {\boldsymbol{#2}}
  {{#2}_{#1}}
}
\newcommand{\Pinv}[1]{{#1}^{+}}

\usepackage{xargs}
\newcommand{\MatrPlus}[3][1={}, 2={}]{
  \ifthenelse{\equal{#1}{} \AND \equal{#2}{}}
  {\boldsymbol{#3}}
  {}
  \ifthenelse{\not{\equal{#1}{}} \AND \equal{#2}{}}
  {\boldsymbol{#3}_{#1}}
  {} 
  \ifthenelse{\equal{#1}{} \AND \not{\equal{#2}{}}}
  {\boldsymbol{#3}^{#2}}
  {}
  \ifthenelse{\not{\equal{#1}{}} \AND \not{\equal{#2}{}}}
  {#3_{#1,#2}}
  {}
}

\newcommand{\TolNonLinear}{\epsilon_{N}}
\newcommand{\TolLinearNewton}{\epsilon}

\newcommand{\LaplSymbol}{L}
\newcommand{\Lapl}[1][]{
  \ifthenelse{\equal{#1}{}}
  {\Matr{\LaplSymbol}}
  {\Matr{\LaplSymbol}[{#1}]}
}

\newcommand{\Damp}{\alpha}
\newcommand{\MinDamp}{\alpha_{0}}

\newcommand{\TolCut}{\delta}
\newcommand{\On}{\mbox{on}}
\newcommand{\Off}{\mbox{off}}

\newcommand{\Diag}[1]{\mbox{diag}\left(#1\right)}

\newcommand{\TolC}{\varepsilon_{\Matr{\Cblock}}}
\newcommand{\Dist}{\mbox{dist}}
\newcommand{\Distances}{D}
\newcommand{\Trace}[1]{\mbox{tr}(#1)}

\newcommand{\Argmin}{\operatornamewithlimits{argmin\vphantom{q}}}
\newcommand{\Argmax}{\operatornamewithlimits{argmax\vphantom{q}}}

\newcommand{\Err}{\ensuremath{\operatorname{err}}}
\newcommand{\ErrTdens}{\Err_{\Vect{\Opt{\Tdens}}}}
\newcommand{\ErrPot}{\Err_{\Vect{\Opt{\Pot}}}}

\DeclareMathOperator{\Grad}{\nabla}

\def\XXint#1#2#3{{\setbox0=\hbox{$#1{#2#3}{\int}$ }
\vcenter{\hbox{$#2#3$ }}\kern-.6\wd0}}

\newcommand{\ABS}[1]{\left|#1\right|}
\newcommand{\Dt}[1]{\partial_{t}#1}

\newcommand{\Cost}{c}

\newcommand{\Plan}{\gamma}

\newcommand{\Lyap}{\mathcal{L}}

\newcommand{\Banach}{\mathcal{X}}

\newcommand{\Ene}{\mathcal{E}}

\newcommand{\InfSupGf}{\Gamma}

\newcommand{\Dim}{d}
\newcommand{\Domain}{\Omega}

\newcommand{\tstep}{k}
\newcommand{\tstepp}{{\tstep+1}}
\newcommand{\Deltat}[1][]{
  \ifthenelse{\equal{#1}{}}
{\Delta t}
  {\Delta t_{#1}}
  }

\newcommand{\MeshPar}{h}
\newcommand{\Cell}[1][]{\ifthenelse{\equal{#1}{}}{T}{T_{#1}}}
\newcommand{\SubCell}[1][]{\ifthenelse{\equal{#1}{}}{t}{t_{#1}}}
\newcommand{\Tsymb}{\mathcal{T}}
\newcommand{\Triang}[1][]
{
  \ifthenelse{\equal{#1}{}}
  {\Tsymb}
  {\Tsymb_{#1}}
}

\newcommand{\PotOp}{\mathcal{U}}
\newcommand{\PotOf}[2][]
{
  \ifthenelse{\equal{#1}{}}
  {\Pot(#2)}
  {\Pot_{#1}(#2)}
}

\newcommand{\RadonPlus}{\mathcal{M}_{+}}

\newcommand{\Lspacechar}{L}
\newcommand{\Lspace}[1]{\Lspacechar^{#1}}

\newcommand{\Cachar}{\mathcal{C}}

\newcommand{\Cont}[1][]{
  \ifthenelse{\equal{#1}{}}
  {\Cachar} 
  {\Cachar^{#1}}
}

\newcommand{\HolderExp}{\delta}
\newcommand{\Holder}[1][]{
  \ifthenelse{\equal{#1}{}}
  {\Cachar^{\HolderExp}}
  {\Cachar^{#1}}  
}

\newcommand{\Lip}[1]
{
  \ifthenelse{\equal{#1}{}}
  {\text{Lip}}
  {\text{Lip}_{#1}}        
}

\newcommand{\Sob}[2][]{
  \ifthenelse{\equal{#2}{}}
  {H^{#1}}
  {W^{#1,#2}}  
}

\newcommand{\DHolder}[2][]
{
  \ifthenelse{\equal{#2}{}}
  {\Cachar^{#1,\HolderExp}}   
  {\Cachar^{#1,#2}}  
}

\newcommand{\VspaceSymb}{\mathcal{V}}
\newcommand{\Vspace}[1][]{
  \ifthenelse{\equal{#1}{}}
  {\VspaceSymb_{\MeshPar}}
  {\VspaceSymb_{\MeshPar}(#1)}
}

\newcommand{\VbaseSymb}{\varphi}
\newcommand{\Vbase}[1][]{
  \ifthenelse{\equal{#1}{}}
  {\VbaseSymb}
  {\VbaseSymb_{#1}}
}
\newcommand{\WspaceSymb}{\mathcal{W}}
\newcommand{\Wspace}[1][]{
  \ifthenelse{\equal{#1}{}}
  {\WspaceSymb_{\MeshPar}}
  {\WspaceSymb_{\MeshPar}(#1)}
}

\newcommand{\WbaseSymb}{\psi}
\newcommand{\Wbase}[1][]{
  \ifthenelse{\equal{#1}{}}
  {\WbaseSymb}
  {\WbaseSymb_{#1}}
}
\newcommand{\Tdens}{\mu}

\newcommand{\OptTdens}{\Tdens^*}

\newcommand{\Pot}{u}

\newcommand{\OptPot}{\Pot^*}

\newcommand{\Cond}{\Tdens}
\newcommand{\GfCond}{\Gfvar}
\newcommand{\Momentum}{y}
\newcommand{\Volt}{\Pot}

\newcommand{\Mgrad}{\Matr{G}}

\newcommand{\Lagr}{\Phi}
\newcommand{\Differential}{\nabla}

\newcommand{\Dflux}{q}

\newcommand{\Trans}{\Psi}
\newcommand{\Gfvar}{\sigma}

\newcommand{\Ingf}[1]{\tilde{#1}}

\newcommand{\Forcing}{b}

\newcommand{\Source}{\Forcing^{+}}
\newcommand{\Sink}{\Forcing^{-}}
\newcommand{\CSource}{f^{+}}
\newcommand{\CSink}{f^{-}}

\newcommand{\Kmax}{\tstep_{\max}}
\newcommand{\Rmax}{\Newi_{\max}}

\newcommand{\Path}{\mathcal{P}}

\newcommand{\Graph}{\mathcal{G}}
\newcommand{\Igraph}{j}
\newcommand{\Nodes}{\mathcal{V}}
\newcommand{\Node}{v}
\newcommand{\Inode}{i}
\newcommand{\Jnode}{j}

\newcommand{\Edges}{\mathcal{E}}
\newcommand{\Iedge}{e}
\newcommand{\Tree}{T}

\newcommand{\BallSymb}{B}
\newcommand{\Ball}[2][]
{
  \ifthenelse{\equal{#2}{}}
  {\BallSymb_{#1}}
  {\BallSymb(#2,#1)}
}
\newcommand{\Avgint}[3][]
{
  \ifthenelse{\equal{#1}{}}
  {({#3})_{#2}}
  {(#3)_{#1,#2}}
}

\newcommand{\OTP}{OTP}

\newcommand{\Newi}{r}
\newcommand{\Newii}{{r+1}}
\newcommand{\TolOpt}{\MyOpt{\epsilon}}

\newcommand{\Scal}[2]{\langle #1,#2\rangle}

\newcommand{\Norm}[2][]
{
  \ifthenelse{\equal{#1}{}}
  {{\left\lVert#2\right\rVert}}
  {{\left\lVert#2\right\rVert}_{#1}}
}

\newcommand{\Fnewtonpot}{f}
\newcommand{\Fnewtongfvar}{g}

\newcommand{\Opt}[1]{#1^{*}}
\newcommand{\MyOpt}[1]{\hat{#1}}

\newcommand{\Ablock}{\mathcal{A}}
\newcommand{\Bblock}{\mathcal{B}}
\newcommand{\Cblock}{\mathcal{C}}

\newcommand{\Zblock}{\mathcal{Z}}
\newcommand{\ZTDblock}{\mathcal{Z}_{\Vect{\Cond}}}
\newcommand{\Vblock}{\mathcal{V}}

\newcommand{\PrecW}{\mathcal{P}_\mathcal{W}}
\newcommand{\PrecV}{\mathcal{P}_\mathcal{V}}

\newcommand{\SchurMC}{\Matr{S}_{\Matr{\Cblock}}}
\newcommand{\SchurMA}{\Matr{S}_{\Matr{\Ablock}}}

\newcommand{\RhsVolt}{f}
\newcommand{\RhsGfvar}{g}

\newcommand{\Diff}{E}

\newcommand{\Wblock}{\mathcal{W}}

\newcommand{\Length}{w}
\newcommand{\Mlength}{\Matr{W}}

\newcommand{\Nnode}{n}
\newcommand{\Nedge}{m}

\newcommand{\Dgfvar}{\Gfvar}
\newcommand{\Jacgf}{\mathcal{J}}

\usepackage{hyperref}       
\usepackage{cleveref}

\newtheorem{Prop}{Proposition}

\crefformat{prob}{Problem~(#2#1#3)}
\Crefformat{prob}{Problem~(#2#1#3)}
\crefmultiformat{prob}{Problems~(#2#1#3)}{ and~(#2#1#3)}{, (#2#1#3)}{ and~(#2#1#3)}
\Crefmultiformat{prob}{Problems~(#2#1#3)}{ and~(#2#1#3)}{, (#2#1#3)}{ and~(#2#1#3)}

\usepackage{epigraph}

\newtheorem{remark}{Remark}
\newtheorem{proof}{Proof}

\usepackage{arxiv}

\usepackage[utf8]{inputenc} 

\usepackage{url}            
\usepackage{amsfonts}       
\usepackage{nicefrac}       
\usepackage{microtype}      

\title{
  Fast Iterative Solution of the Optimal Transport Problem on
  Graphs
}

\author{
  Enrico Facca\\
  Centro di Ricerca Matematica Ennio De Giorgi,\\
  Scuola Normale Superiore,\\
  Piazza dei Cavalieri, 3, 56126 Pisa, Italy \\
  \texttt{enrico.facca@sns.it}
  \And
  Michele Benzi \\
 Scuola Normale Superiore,\\
 Piazza dei Cavalieri, 7, 56126 Pisa, Italy  \\
 \texttt{michele.benzi@sns.it}.
}

\begin{document}

\maketitle

\begin{abstract}
    In this paper, we address the numerical solution of the Optimal
  Transport Problem on undirected weighted graphs, taking the shortest
  path distance as transport cost.  The optimal solution is obtained
  from the long-time limit of the gradient descent dynamics.  Among
  different time stepping procedures for the discretization of this
  dynamics, a backward Euler time stepping scheme combined with the
  inexact Newton-Raphson method results in a robust and accurate
  approach for the solution of the Optimal Transport Problem on
  graphs.  It is found experimentally that the algorithm requires
  solving between $\mathcal{O}(1)$ and $\mathcal{O}(\Nedge^{0.36})$
  linear systems involving weighted Laplacian matrices, where $\Nedge$
  is the number of edges.  These linear systems are solved via
  algebraic multigrid methods, resulting in an efficient solver for
  the Optimal Transport Problem on graphs.

\end{abstract}


\section{Introduction}
\label{sec:intro}
The Optimal Transport Problem (OTP) is a type of optimization problem
where the goal is is to determine the optimal reallocation of
resources from one configuration to another.  The
OTP is also known as the Monge-Kantorovich Problem, after
Gaspard Monge, who proposed the first formulation of the problem
in~\cite{Monge:1781}, and Leonid Kantorovich, who introduced
in~\cite{Kantorovich:1942} the relaxed formulation studied nowadays.

In a rather general framework, the Monge-Kantorovich Problem can be
formulated as follows.  Consider a measurable space $\Banach$, two
non-negative measures $\CSource$ and $\CSink$ with equal mass, and a
cost function $\Cost:\Banach\times \Banach\to \REAL$ such that
$\Cost(x,y)$ describes the cost paid for transporting one unit of mass
from $x$ to $y$.  We want to find the optimal plan $\Opt{\Plan}$ in
the space of non-negative measures on the product
space $\Banach\times\Banach$ (which we denote with $\RadonPlus(\Banach
\times \Banach)$ ) that solves the following minimization problem:
\begin{equation}
  \label[prob]{eq:otp}
  \inf_{\Plan\in \RadonPlus(\Banach \times \Banach)}
  \left\{ 
    \int_{\Banach\times\Banach} \Cost(x,y) \mathrm{d}\Plan(x,y)
    \ : \
    \begin{gathered}
      \mbox{for all } A,B \mbox{ measurable subsets of } \Banach\\
      \Plan(A,\Banach) = \CSource(A) \quad
      \Plan(\Banach,B) = \CSink(B)  
    \end{gathered} 
  \right\}\;  .
\end{equation}
In recent years, a number of authors have contributed to the
 development of the \OTP\ theory (see for example
\cite{Santambrogio:2015,Ambrosio:2003,Villani:2008}), and several
reformulations of this problem have been found, together with
many connections with apparently unrelated areas.  More
recently, the mathematical ``tools'' provided by the OTP started to be
used in more applied settings. An example of application is the use of
the Wasserstein distance (which is the value obtained in the
minimization~\cref{eq:otp}) in inverse
problems~\cite{Yang-Engquist-et-al:2018,Metivier-et-al:2016}, image
processing~\cite{Lellmann-et-al:2014,Haker-et-al:2004}, and machine
learning~\cite{Arjovsky-et-al:2017}. However, the use of these
``tools'' is hindered by the high computation cost of the numerical
solution of the OTP.

In this paper we  focus on the case where the ambient
space $\Banach$ of the \OTP\ is a connected, weighted, undirected,
graph $\Graph=(\Nodes,\Edges)$ with $\Nnode$ vertices and with
$\Nedge$ edges with strictly positive weights
$\Vect{\Length}\in\REAL^{\Nedge}$.  In particular, we focus on sparse
graphs, thus $\Nedge=\mathcal{O}(\Nnode)$.  We consider the cost
function $\Cost$ given by the shortest weighted path metric on the
graph $\Graph$ which, given $\Node_{i},\Node_{j}\in \Nodes$ is defined
as
\begin{equation*}
  \Cost(\Node_{i},\Node_{j})
  =
  \Dist_{\Graph,\Vect{\Length}}(\Node_{i},\Node_{j})
  =\inf_{\Path(\Node_{i},\Node_{j})} \left\{
    \sum_{\Iedge\in \Path(\Node_{i},\Node_{j})} \Vect[\Iedge]{\Length}
    \ : \ 
    \Path(\Node_{i},\Node_{j})
    =
    \mbox{ path in $\Graph$ connecting $\Node_{i}$ with $\Node_{j}$}
  \right\},
\end{equation*}
a problem often referred to as $L^1$-\OTP.  As an
  example of application, the Wasserstein distance on graphs has been
  used in~\cite{Lin2011} to introduce an analog of Ricci curvature on
  graphs. This notion has been then used
  in~\cite{sia2019ollivier,ni2019community} in the study of community
  detection on graphs.

In the graph setting the transported measures $\CSource$ and $\CSink$
of the OTP are simply two vectors $\Vect{\Source},\Vect{\Sink}\in
\REAL_{+}^{\Nnode}=[0,+\infty[^{\Nnode}$ such that
    $\sum_{\Inode=1}^{\Nnode}\Vect[\Inode]{\Source}=\sum_{\Inode=1}^{\Nnode}\Vect[\Inode]\Sink$.
    Then \cref{eq:otp} becomes the following linear-programming
    problem:
\begin{equation}
  \label[prob]{eq:otp-graph}
  \inf_{\Matr{\Plan} \in \REALP^{\Nnode,\Nnode}} \left\{
  \Trace{\Matr{\Distances}^T\Matr{\Plan}} \ : \
    \begin{aligned}
      \Matr{\Plan} \Vect{1}
      &= 
      \Vect{\Source}\\
      \Vect{1} \Matr{\Plan}^T
      &= 
      \Vect{\Sink}
    \end{aligned}
  \right\}\;  ,
\end{equation}
where $\Matr{\Distances}$ is the distance matrix of the graph $\Graph$
i.e., the symmetric matrix whose $(i,j)$ entry is equal to
  $\Dist_{\Graph,\Vect{\Length}}(\Node_{i},\Node_{j})$.  We are now
looking for an optimal matrix $\Matr{\Opt{\Plan}}\in
\REAL^{\Nnode,\Nnode}_{+}$ (in general there may exist multiple
solutions) in which entry $\MatrPlus[i]{j}{\Opt{\Plan}}$ describes the
portion of ``mass'' optimally transported from the node $\Node_i$ to
node $\Node_j$. Unfortunately, since the number of unknowns scales
with $\mathcal{O}(\Nnode^2)$, \cref{eq:otp-graph} becomes intractable
numerically for graphs with a large number of nodes.
Moreover, in the graph setting, there is the additional cost of
computing the distance matrix $\Matr{\Distances}$.

On the other hand, when $\Cost=\Dist_{\Graph,\Vect{\Length}}$, the OTP can
be rewritten in the equivalent form of an $\ell^1$-minimal
flow problem, given by
\begin{equation}
  \label[prob]{eq:beckmann-intro}
  \min_{\Vect{\Dflux}\in\REAL^{\Nedge}} 
  \left\{
    \sum_{\Iedge=1}^{\Nedge} 
    \Vect[\Iedge]{\Length} |\Vect[\Iedge]{\Dflux}|
    \ 
    \mbox{s.t.} 
    \ 
    \Matr{\Diff} \Vect{\Dflux} = \Vect{\Source}-\Vect{\Sink}
  \right\},
\end{equation}
where $\Matr{\Diff}$ is the signed incidence matrix of the graph
$\Graph$.  This problem has long been known in the study of network
flows as the minimum-cost flow problem with no edge
capacities. Different approaches like cycle canceling, network
simplex, and the Ford-Fulkerson algorithms have been proposed to solve
it (see~\cite{ahuja1993network} for a complete overview).  More
recently, the numerical solution of~\cref{eq:beckmann-intro} has been
addressed via interior point methods in~\cite{Madry:2013}, where it is
shown that the problem can be solved on bipartite graphs by solving
$\mathcal{O}(\Nedge^{3/7\approx 0.428})$ weighted Laplacian systems.
The numerical solution of~\cref{eq:beckmann-intro} as OTP on a
graph has been studied in~\cite{Ling-et-al:2007} using the simplex
algorithm. More recently, \cite{Essid2018} considered a quadratically
regularized version of~\cref{eq:beckmann-intro} in order to deal with
the non-uniqueness of the minimal flow.

In this paper, we consider a different approach where we obtain a
solution of~\cref{eq:beckmann-intro} by means of an equivalent
formulation described in~\cite{Facca-discrete:2018}, as a minimization
problem for an energy functional
$\Lyap(\Vect{\Cond}):\REALP^{\Nedge}\to \REAL $ with
$\Vect{\Tdens}\in\REALP^{\Nedge}$. Here $\Vect{\Tdens}$ can be
interpreted as a conductivity associated to the edges of the graph.
In this paper, a minimum of $\Lyap$ is sought via a gradient
descent approach, not applied directly to the functional $\Lyap$, but
rather to its composition with the change of variable
$\Trans:\REAL^{\Nedge}\to \REALP^{\Nedge}$ that, acting
component-wise, gives
$\Vect{\Cond}=\Trans(\Vect{\Gfvar})=\Vect{\Gfvar}^2/4$. While
minimizing $\Ingf{\Lyap}=\Lyap\circ \Trans$ is clearly equivalent to
minimizing $\Lyap$, the first shows a much nicer convergence behavior.

The gradient descent approach applied to minimization of
$\Ingf{\Lyap}(\Vect{\Gfvar})$ may be seen as finding the vector
$\Vect{\Opt{\Gfvar}}=\lim_{t\to +\infty} \Vect{\Gfvar}(t)$, where
$\Vect{\Gfvar}(t)$ solves the initial value problem
\begin{equation*}
  \Dt{\Vect{\Gfvar}}(t)=-\Grad
  \Ingf{\Lyap}\left(\Vect{\Gfvar}(t)\right), \quad \Vect{\Gfvar}(0)=\Vect{\Gfvar}_0 .
\end{equation*}
We compare in terms of accuracy and efficiency three approaches for the
discretization of this ODE: the forward Euler time stepping (the
classical gradient descent algorithm), the accelerated gradient
descent method described in~\cite{Nesterov:1983}, and the backward
Euler time stepping. The last approach results in solving a sequence
on non-linear equations given by
\begin{equation*}
  \Vect{\Gfvar}^{\tstepp} = 
  \Vect{\Gfvar}^{\tstep} -
  \Deltat[\tstepp]\Grad \Ingf{\Lyap} (\Vect{\Gfvar}^{\tstepp}),
  \quad \tstep=0,1,\ldots
\end{equation*}
iterated until convergence to a steady state
configuration. The above non-linear equation is solved via an inexact
damped Newton-Raphson method, that requires the solution of a
sequence of saddle point linear systems in the form
\begin{equation*}
  \begin{pmatrix}
    \Matr{\Ablock} & \Matr{\Bblock}^T
    \\
    \Matr{\Bblock} & -\Matr{\Cblock}
  \end{pmatrix}
  \begin{pmatrix}
    \Vect{x}
    \\
    \Vect{y}
  \end{pmatrix}
  = 
  \begin{pmatrix}
     \Vect{\Fnewtonpot}
     \\ 
     \Vect{\Fnewtongfvar}
   \end{pmatrix}\;,
\end{equation*}
with $ \Matr{\Ablock},\Matr{\Bblock},\Matr{\Cblock}$ sparse matrices.
In contrast, the forward Euler and the Accelerated Gradient descent
both require the solution of only one weighted Laplacian system per
step. However, the additional cost is off-set by the much smaller
number of time steps and linear systems to be solved to reach a steady
state, hence the backward Euler method turns out to be the most efficient
among those considered.  We examine different approaches for the
solution of the saddle point linear systems. The simplest approach,
which exploits matrix $\Matr{\Cblock}$ being diagonal, turns out to be
the most efficient, even if it may require using a smaller time step
size $\Deltat[\tstep]$, thus more time steps.  In fact, it only
requires the solution of a sequence of linear systems involving
weighted Laplacian matrices, which can be solved via algebraic
multigrid methods.  Moreover, the efficiency of this approach can be
drastically improved removing along the time evolution
removing those edges (and the corresponding nodes, if
isolated) with conductivity $\Vect{\Tdens}=\Trans(\Vect{\Gfvar})$
below a given small threshold.

The resulting method has been tested on different graphs and vectors
$\Vect{\Forcing}$. In the numerical experiments where exact solutions
are available, the numerical method proposed shows a high level of
accuracy.  The total number of linear systems to be solved, which
accounts for almost all of the computation effort and which coincides
with the total number of Newton steps, has been found experimentally
to range between being constant with respect to number of edges to
scaling like $\mathcal{O}( \Nedge ^{0.36})$ on graphs generated using
the Erd\"os-R\'enyi, the Watts-Strogatz, and the Barabasi-Albert
models.

\section{OTP on graphs with cost$=\Dist_{\Graph,\Vect{\Length}}$}
\label{sec:otp-graph}

Before presenting the functional $\Lyap$ we want to
minimize, we need to introduce the equivalent formulations
of the OTP addressed in this paper. To this aim,
let us introduce some definitions and fix some
notations. First, we define the vector
\begin{equation*}
  \label{eq:forcing}
  \Vect{\Forcing}=\Vect{\Source}-\Vect{\Sink}\; ,
\end{equation*} 
which will be referred to as \emph{forcing term}.
Then,
fixing an orientation on the graph $\Graph$ (that will have
no influence in our formulation) we introduce matrices
$\Matr{\Diff}\in\REAL^{\Nedge,\Nnode}, \Mlength
\in\REAL^{\Nedge,\Nedge}$ and $\Mgrad \in
\REAL^{\Nnode,\Nedge}$ defined as
\begin{equation}
  \nonumber
  \label{eq:sign-incidence}
  \Matr[\Iedge,\Inode]{\Diff} = 
  \left\{
    \begin{aligned}
      &1  &&\mbox{ if } \Inode =\mbox{"head" of edge }\Iedge\\
      &-1 &&\mbox{ if } \Inode =\mbox{"tail" of edge }\Iedge\\
      &0  &&\mbox{ otherwise} 
    \end{aligned}
  \right.,
  \quad
  \Mlength=\Diag{\Vect{\Length}} ,
  \quad
  \Mgrad=\Mlength^{-1} \Matr{\Diff}^T.
\end{equation}
Matrix $\Matr{\Diff}$, which is the signed incidence matrix of the
graph $\Graph$, and matrix $\Mgrad$ play the role of the
(negative) divergence and gradient operators in the weighted
graphs setting.  Moreover, for any \emph{non-negative}
``conductivity'' $\Vect{\bar{\Cond}}\in \REALP^{\Nedge}$, we will
denote with
\begin{equation}
  \nonumber
  \label{eq:weighted-laplcian}
  \Lapl[\Vect{\bar{\Tdens}}]:=
  \Matr{\Diff} \Diag{\Vect{\bar{\Tdens}}} \Mgrad
\end{equation}
the $\Vect{\bar{\Cond}}$-weighted Laplacian matrix of the graph
$\Graph$, and with
\[
\Vect{\PotOp}\Of{\Vect{\bar{\Cond}}} :=
\Pinv{\Lapl[\Vect{\bar{\Cond}}]}\Vect{\Forcing},
\]
the ``potential'' induced by the ``conductivity''
$\Vect{\bar{\Cond}}$ and the forcing term $\Vect{\Forcing}$.  Here
and in the following, the symbol $\Pinv{}$ denotes the Moore-Penrose
generalized inverse. The action of the pseudo-inverse on a given
vector is found computing the minimum norm solution of the
corresponding (singular) linear system. Note that, since the graph
$\Graph$ is connected and the entries of the forcing term vector
$\Vect{\Forcing}$ sum to zero, if $\Vect{\bar{\Cond}}$ is strictly
positive then $\Vect{\PotOp}\Of{\Vect{\bar{\Cond}}}$ is always well
defined. If some entries of $\Vect{\bar{\Cond}}$ are null, the
forcing term $\Vect{\Forcing}$ must belong to range of
$\Lapl[\Vect{\bar{\Cond}}]$ for $\Vect{\PotOp}\Of{\Vect{\bar{\Cond}}}$
to be well defined.\\

\paragraph{Equivalent formulations of OTP with cost
  $\Dist_{\Graph,\Vect{\Length}}$} 
The first equivalent formulation of the OTP with cost
$\Dist_{\Graph,\Vect{\Length}}$ is given by the following minimal-cost flow 
problem:
\begin{equation}
  \label[prob]{eq:beckmann}
  \min_{\Vect{\Dflux}\in\REAL^{\Nedge}} 
  \left\{
    \Norm[1]{\Mlength \Vect{\Dflux}}
    \ 
    \mbox{s.t.} 
    \ 
    \Matr{\Diff} \Vect{\Dflux} = \Vect{\Forcing}
  \right\},
\end{equation}
where $\Norm[1]{\cdot}$ denotes the $\ell^1$-norm.  In an electrical
flow analogy, we are looking for an optimal current
$\Vect{\Opt{\Dflux}}$ that minimizes the overall total flux on the
graph, measured with a weighted $\ell^1$-norm, under the constraint
that the flux satisfies the Kirchhoff laws $\Matr{\Diff} \Vect{\Dflux}
=\Vect{\Forcing}$.

The second equivalent problem, which is the dual
of \cref{eq:beckmann}, is the
following maximization problem:
\begin{equation}
  \label[prob]{eq:dual}
  \max_{\Vect{\Volt}\in\REAL^{\Nnode}} 
  \left\{
    \Vect{\Forcing}^T
    \Vect{\Volt}
    \ 
    \mbox{s.t.} 
    \ 
    \Norm[\infty]{\Mgrad \Vect{\Volt}}
    \leq 1
\right\},
\end{equation}
where $\Norm[\infty]{\cdot}$ denotes the $\ell^{\infty}$-norm.  This
last problem is actually a simplified version of the dual problem of
the primal~\cref{eq:otp-graph} (see~\cite{Santambrogio:2015} for the
detailed proof). \Cref{eq:beckmann,eq:dual} may have multiple
solutions.

A third equivalent formulation, which subsumes
\cref{eq:beckmann,eq:dual}, and which will be the one we actually
solve in this paper, can be introduced by
rewriting~\cref{eq:dual} as follows:
\begin{align}
  \label{eq:dual-original}
  &\max_{\Vect{\Volt}\in\REAL^{\Nnode}} 
  \left\{
    \Vect{\Forcing}^T
    \Vect{\Volt}
    \ 
    \mbox{s.t.} 
    \
    |(\Mgrad \Vect{\Volt})_{\Iedge}|\leq 1 
    \ 
    \forall \Iedge=1,\ldots,\Nedge
    \right\}
    \; ,
  \\
  \label{eq:dual-transformed}
  \equiv
  &\max_{\Vect{\Volt}\in\REAL^{\Nnode}} 
  \left\{
  \Vect{\Forcing}^T
  \Vect{\Volt}
  \ 
  \mbox{s.t.} 
  \
  \frac{\Vect[\Iedge]{\Length}}{2}
  \left(|(\Mgrad \Vect{\Volt})_{\Iedge}|^2-1\right)\leq 0 
  \ 
  \forall \Iedge=1,\ldots,\Nedge
  \right\}
  \; ,
\end{align}
where we first square both sides of the inequality constraints
in~\cref{eq:dual-original}, and then we multiply by the positive
factor $\Vect[\Iedge]{\Length}/2$ (the reason for these
transformations will be clear later).  Now, we can introduce a
non-negative Lagrange multiplier $\Vect{\Tdens}\in \REALP^\Nedge$ for
the inequality constraints in~\cref{eq:dual-transformed}, so that
\cref{eq:dual,eq:dual-transformed} become equivalent to the following
inf-sup problem:
\begin{equation}
  \label[prob]{eq:min-max}
  \inf_{\Vect{\Cond}\in\REALP^\Nedge}\sup_{\Vect{\Volt}\in\REAL^{\Nnode}}
  \Lagr(\Vect{\Volt},\Vect{\Cond}) :=  
  \Vect{\Forcing}^T \Vect{\Volt}  
  - \frac{1}{2}\Vect{\Volt}^T \Lapl[\Vect{\Tdens}] \Vect{\Volt} 
  + 
  \frac{1}{2} \Vect{\Cond}^{T}  \Vect{\Length}
  \; ,
\end{equation}
with Lagrangian $\Lagr$.
If we now introduce the functional
\begin{equation}
  \nonumber
  \label{eq:joule}
  \Ene(\Vect{\Cond}) 
  =
  \sup_{\Vect{\Volt} \in \REAL^{\Nnode} } 
  \left( 
    \Vect{\Forcing}^T \Vect{\Volt}  
    - \frac{1}{2}\Vect{\Volt}^T \Lapl[\Vect{\Tdens}] \Vect{\Volt} )
  \right)
  =
  \frac{1}{2}\Vect{\Forcing}^T\Vect{\PotOp}\Of{\Vect{\Cond}}
  =
  \frac{1}{2}
  \Vect{\Forcing}^T
  \Pinv{\Lapl}[\Vect{\Tdens}]
  \Vect{\Forcing}\; ,
\end{equation}
which is the Joule dissipated energy associated to $\Vect{\Cond}$, we
may look at~\cref{eq:min-max} only in terms of the variable
$\Vect{\Cond}$. Hence, we want to find the optimal
$\Vect{\Opt{\Cond}}$ solving the following minimization problem:
\begin{equation}
  \label[prob]{eq:min-lyap}
  \min_{\Vect{\Cond}\in\REALP^\Nedge}
  \left\{
    \Lyap(\Vect{\Cond})   
    =
    \Ene(\Vect{\Cond})
    + 
    \frac{1}{2}   \Vect{\Length}^{T}\Vect{\Cond}
  \right\} \; ,
\end{equation}
and then setting
$\Vect{\Opt{\Pot}}=\Vect{\PotOp}\Of{\Vect{\Opt{\Cond}}}$.  Again, in
the electrical network analogy where $\Vect{\Source}$ and
$\Vect{\Sink}$ describe the electrical charge distributions, the above
problem amounts to finding the optimal conductivity
$\Vect{\Opt{\Cond}}$ giving the best trade-off between the Joule
dissipated energy $\Ene(\Vect{\Cond})$ and the ``infrastructure cost''
$\frac{1}{2} \Vect{\Cond}^{T} \Vect{\Length}$. \Cref{eq:min-max} is
actually a min-max problem since functional $ \Lyap(\Vect{\Cond})$
admits a minimum $\Vect{\Opt{\Cond}}$, being convex and since
$\lim_{\Norm{\Vect{\Cond}}\to
  +\infty}\Lyap(\Vect{\Cond})=+\infty$. Then, given a minimizer
$\Vect{\Opt{\Cond}}$, such maximization of the Lagrangian $\Lagr$ of
\cref{eq:min-max} is equivalent to finding
$\Vect{\Opt{\Volt}}=\Vect{\PotOp}\Of{\Vect{\Opt{\Cond}}}$.  Any vector
$\Vect{\Opt{\Volt}}$  also solves~\cref{eq:dual}.

The relation among the different optimization problems presented in
this section are described in the following proposition.
\begin{Prop}
  \label{prop:equiv}
  \Cref{eq:beckmann,eq:dual,eq:min-lyap} are equivalent. This means
  that the following equalities hold:
  \begin{equation}
    \nonumber
    \label{eq:graph-otp}
    \max_{\Vect{\Volt}\in\REAL^{\Nnode}}
    \left\{
    \Vect{\Forcing}^T\Vect{\Pot}
    \ : \ 
    \Norm{\Mgrad\Vect{\Volt}}_{\infty}\leq 1
    \right\}
    =    
    \min_{\Vect{\Dflux}\in\REAL^{\Nedge}}
    \left\{
    \Norm[1]{\Mlength \Vect{\Dflux}}
    \ :\ 
    \Matr{\Diff}\Vect{\Dflux}=\Vect{\Forcing}
    \right\}
    =    
    \min_{\Vect{\Tdens}\in\REAL^{\Nedge}\geq 0} 
    \Lyap(\Vect{\Tdens})
    \; .
  \end{equation}
  Moreover, given a solution $\Opt{\Vect{\Cond}}$ that minimizes the
  functional $\Lyap$ we can recover a maximizer $\Opt{\Vect{\Volt}}$
  of~\cref{eq:dual} by solving
  \begin{subequations}
    \label{eq:bp}
    \begin{align}
      &  \Lapl[\Vect{\Opt{\Cond}}] \Vect{\Opt{\Volt}}=\Vect{\Forcing} 
      \label{eq:bp-div}
      \; ,
      \\
      &|(\Mgrad \Vect{\Opt{\Volt}})_{\Iedge}|\leq 1 
      \quad \forall \Iedge =1,\ldots, \Nedge
      \label{eq:bp-bound}
      \; ,\\
      &|(\Mgrad \Vect{\Opt{\Volt}})_{\Iedge}|=1 
      \quad \Opt{\Vect[\Iedge]{\Cond}}>0
      \label{eq:bp-eikonal}
    \end{align}
  \end{subequations}
  (the KKT equations for~\cref{eq:dual}).
  A minimizer $\Opt{\Vect{\Dflux}}$
  of~\cref{eq:beckmann} can be represented in the form
  \begin{equation}
    \label{eq:optdflux}
    \Opt{\Vect{\Dflux}} 
    =
    \Diag{\Opt{\Vect{\Cond}}}\Mgrad\Opt{\Vect{\Volt}}ù
    \; .
  \end{equation}
  Moreover, the following equality holds:
  \begin{equation}
    \nonumber
    \label{eq:discrete-minima}
    \Vect{\Opt{\Cond}}=|\Opt{\Vect{\Dflux}}|.
  \end{equation}
\end{Prop}
The proof of the equivalence between~\cref{eq:beckmann,eq:dual} is a
straightforward application of standard duality results
(see~\cite[Section 4, Chapter3]{Ekeland-Teman:1999}).  We refer the
reader to~\cite{Facca-discrete:2018} for the proof of the equivalence
with the problem of minimizing $\Lyap$.  The above proposition states that any optimal flux
$\Vect{\Opt{\Dflux}}$ is induced by the gradient of the potential
$\Vect{\Opt{\Volt}}$, multiplied by the optimal conductivity
$\Vect{\Opt{\Cond}}$.  Note that, for any pair
$(\Vect{\Dflux},\Vect{\Pot})\in \REAL^{\Nedge}\times \REAL^{\Nnode}$
satisfying the constraints $\Matr{\Diff}\Vect{\Dflux}=\Vect{\Forcing}$
and $\Norm{\Mgrad\Vect{\Volt}}_{\infty}\leq 1$, the following
inequality holds:
\begin{equation}
  \label{eq:duality-gap}
  \Gap(\Vect{\Dflux},\Vect{\Pot}):=\Norm[1]{\Mlength \Vect{\Dflux}}- \Vect{\Forcing}^T\Vect{\Pot}\geq 0\; ,
\end{equation}
where equality guarantees the optimality of
$(\Vect{\Dflux},\Vect{\Pot})\in \REAL^{\Nedge}\times \REAL^{\Nnode}$.
The quantity $\Gap(\Vect{\Dflux},\Vect{\Pot})$ is called \emph{Duality
  gap}.

It is important to note that any optimal solution $\Vect{\Opt{\Tdens
}}$ may be zero on several edges, and there may be nodes where the
corresponding rows of the matrix $\Lapl[\Vect{\Opt{\Cond}}]$ is null,
leading to a singular but compatible system of equations. This means
that there may exist a large null-space of $\Lapl[\Vect{\Opt{\Cond}}]$
that may be added to any admissible solution $\Opt{\Pot}$ (an optimal
potential), as long as the constraints in~\cref{eq:bp-bound} are
satisfied.  However, it is easy to identify a base for this subspace,
given by the ``support'' vectors of the node ``surrounded'' by edges
with zero conductivity.  Uniqueness of the solution of the primal
problem may fail even in case when the dual admits a unique (up to a
constant) solution (for example, see Test-case~1
in~\cref{sec:numerical-experiments}).

All the discrete problems mentioned in this section find
their counterpart in the OTP posed in a domain
$\Domain\subset \REAL^{\Dim}$ with cost equal to the
Euclidean distance.  We refer to~\cite{Facca-discrete:2018}
for a detailed description of these connections between discrete
and continuous OT problems.

\section{ OTP solution via dynamical approaches}
\label{sec:otp-physarum}
Among the different formulations of the OTP presented in the previous
section, in this paper we will address the problem of finding a
solution $\Vect{\Opt{\Cond}}$ of the
minimization~\cref{eq:min-lyap}. The solution $\Vect{\Opt{\Volt}}$
of~\cref{eq:dual} will be derived from the equivalent min-max
\cref{eq:min-max}, while the $\Vect{\Opt{\Dflux}}$ solution
of~\cref{eq:beckmann} will be given by~\cref{eq:optdflux}.

In this paper, the problem of minimizing $\Lyap$ is solved via
a gradient descent approach, where we introduced a one-to-one
transformation $\Trans: \REAL^{\Nedge} \to \REALP^{\Nedge}$ , acting
component-wise, with
\begin{equation}
\nonumber      
\label{eq:trans}
\Vect[\Iedge]\Cond=\Trans(\Vect[\Iedge]{\GfCond})
\;.
\end{equation}
It is clear that any minimum $\Vect{\Opt{\Cond}}$ of the functional
$\Lyap$ can be recovered from any minimum $\Vect{\Opt{\GfCond}}$ of
the composed functional
\begin{equation}
  \nonumber
  \label{eq:lyap-gfvar}
  \Ingf{\Lyap}(\Vect{\GfCond})=\Lyap(\Trans(\Vect{\GfCond}))=
  \frac{1}{2}\left(
  \Vect{\Forcing}^T\Pinv{\Lapl}[\Trans(\Vect{\Gfvar})]\Vect{\Forcing}+
  \Vect{\Length}^T\Trans(\Vect{\Gfvar}) \right)
  \; ,
\end{equation}
and vice-versa. We now differentiate with respect to
$\Vect[\Iedge]{\Gfvar}$ the functional
$\Ingf{\Lyap}(\Vect{\Gfvar})$, obtaining the following expression:
\begin{equation*}
  \partial_{\Vect[\Iedge]{\Gfvar}}
  \Ingf{\Lyap}\left(\Vect{\Gfvar}\right)
  =
  \partial_{\Vect[\Iedge]{\Tdens}}
  \left(
  \Lyap
  \left(
  \Vect{\Tdens}
  \right)
  \right)
  \Trans'(\Vect[\Iedge]{\Gfvar})
  =
  \frac{1}{2}
  \left(
  \Vect{\Forcing}^{T}
  \partial_{\Vect[\Iedge]{\Tdens}}
  \left(
  \Pinv{\Lapl}[\Vect{\Tdens}]
  \right)
  \Vect{\Forcing}
  +
  \Vect[\Iedge]{\Length}
  \right)
  \Trans'(\Vect[\Iedge]{\Gfvar})
  \;.
\end{equation*}
Using the following matrix identity (see~\cite[Theorem 4.3.]{GoPe:1973})
\begin{multline*}
  \partial_{\Vect[\Iedge]{\Tdens}}
  \left(
  \Pinv{\Lapl}[\Vect{\Tdens}]
  \right) =
  -
  \Pinv{\Lapl}[\Vect{\Tdens}]
  \partial_{\Vect[\Iedge]{\Tdens}}
  \left(\Lapl[\Vect{\Tdens}]\right)
  \Pinv{\Lapl}[\Vect{\Tdens}]
  +
  \Pinv{\Lapl}[\Vect{\Tdens}]
  \Pinv{\Lapl}[\Vect{\Tdens}]^{T}
  \partial_{\Vect[\Iedge]{\Tdens}}
  \left(\Lapl^T[\Vect{\Tdens}]\right)  
  (\Matr{I}-\Lapl[\Vect{\Tdens}] \Pinv{\Lapl}[\Vect{\Tdens}])
  \\
  +
  (\Matr{I}-\Pinv{\Lapl}[\Vect{\Tdens}]\Lapl[\Vect{\Tdens}] )
  \partial_{\Vect[\Iedge]{\Tdens}}
  \left(\Lapl^T[\Vect{\Tdens}]\right)
  \Pinv{\Lapl}[\Vect{\Tdens}]^{T}
  \Pinv{\Lapl}[\Vect{\Tdens}]\;,
\end{multline*}
and the fact that $\left(\Matr{I}-\Lapl[\Vect{\Tdens}]
\Pinv{\Lapl}[\Vect{\Tdens}]\right)\Vect{\Forcing}=\Vect{0}$, we obtain
the following equation for the gradient of the functional
$\Ingf{\Lyap}(\Vect{\Gfvar})$:
\begin{equation}
  \label{eq:gradient-lyap}
  \Grad \Ingf{\Lyap}\left(\Vect{\Gfvar}\right) = \Vect{\Length}
  \odot \Trans'\left(\Vect{\Gfvar}\right) \odot \frac{1}{2} \left(
  -\left(\Mgrad
  \Vect{\PotOp}\Of{\Trans(\Vect{\Gfvar})}\right)^2+\Vect{1}
  \right)\; . 
\end{equation}
Similar computations in the evaluation of the gradient of
$\Lyap(\Vect{\Tdens})$ can be found
in~\cite{Facca-discrete:2018,Bonifaci:2019}.

In this paper, we use the transformation
\begin{equation}
  \label{eq:trans-square}
  \Trans(\Vect{\Gfvar})=\frac{\Vect{\Gfvar}^2}{4}\;,
\end{equation}
and we set the gradient descent dynamics in the space $\REAL^{\Nedge}$
endowed with the scalar product
$\Scal{\Vect{x}}{\Vect{y}}_{\Vect{\Length}} :=
\Vect{x}^{T}\Diag{\Vect{\Length}}\Vect{y}$ and the norm
$\Norm[\Vect{\Length}]{x}:=\sqrt{\Scal{\Vect{x}}{\Vect{x}}_{\Vect{\Length}}}$,
which is introduced to remove the vector $\Vect{\Length}$ from the
right-hand side of~\cref{eq:gradient-lyap}.  Given
$\Vect{\Gfvar}_0>0$, the corresponding initial value problem is
given by
\begin {subequations}
  \label{eq:gf-ode} 
  \begin{align}
    \label{eq:gf-ode-dyn}
    \Dt{\Vect{\Gfvar}(t)}
    &=
    -\Differential_{\Vect{\Gfvar}} \Ingf{\Lyap}(\Vect{\Gfvar})
    =
    \Vect{\Gfvar}(t)
    \odot
    \frac{1}{4}
    \left(
    \left(\Mgrad \Vect{\PotOp}\Of{\Trans(\Vect{\Gfvar})}\right)^2-\Vect{1}
    \right)
    \; ,
    \\
    \label{eq:pp-ode-bc}
    \Vect{\Gfvar}(0)&=\Vect{\Gfvar}_0
    \; .
  \end{align}
  \end{subequations}
A minimizer $\Vect{\Opt{\Gfvar}}$ of functional
$\Ingf{\Lyap}(\Vect{\Gfvar})$ is sought as the long-time limit
of the solution $\Vect{\Gfvar}(t)$ of~\cref{eq:gf-ode}.  Solutions
of~\cref{eq:min-lyap,eq:bp,eq:beckmann} are then given by 
\[
\Vect{\Opt{\Cond}}=\Trans(\Vect{\Opt{\Gfvar}}), 
\quad 
\Vect{\Opt{\Pot}}=\Vect{\PotOp}\Of{\Vect{\Opt{\Cond}}}
\quad 
\Vect{\Opt{\Dflux}}=\Vect{\Opt{\Cond}}\odot \Mgrad \Vect{\Opt{\Pot}}.
\]

The change of variable $\Trans$ in~\cref{eq:trans-square} is
introduced in order to recover the following dynamical system:
\begin {subequations}
  \label{eq:pp-ode} 
  \begin{align}
    \label{eq:pp-ode-dyn}
    &
    \Dt{\Vect{\Cond}}(t)= 
    \Vect{\Cond}(t)
    \odot
    \left(
    \left( \Mgrad \Vect{\PotOp}\Of{\Vect{\Cond}(t)}\right)^2
    -\Vect{1}
    \right)
    =
    2
    \Vect{\Length}^{-1}
    \odot
    \Vect{\Cond}(t)
    \odot
    \left(-\Grad_{\Vect{\Cond}}\Lyap(\Vect{\Cond}(t))\right)
    \; ,
    \\
    \label{eq:pp-ode-bc}
    &
    \Vect{\Cond}(0)=\Vect{\Cond}_0>\Vect{0}.
  \end{align}
\end{subequations}
The steady state solution
$\Vect{\Opt{\Cond}}=\lim_{t\to+\infty}\Vect{\Cond}(t)$
of~\cref{eq:pp-ode} and its variants have been related to different
types of optimization problems, like the shortest path
problem~\cite{Tero-et-al:2007, bonifaci:physarum}, the Basis Pursuit
Problem~\cite{Facca-discrete:2018,Bonifaci:2019} and different
transport problems in a continuous setting
\cite{Facca-et-al:2018,Facca-et-al-branch:2018}.
In~\cite{Bonifaci:2019} it has been shown that the dynamical system
in~\cref{eq:pp-ode} can be interpreted as mirror descent dynamics for
the functional $\Lyap$, where the descent direction is given by
  the product of the vector
  $-\Grad_{\Vect{\Cond}}\Lyap(\Vect{\Cond}(t))$ and the vector
  $\Vect{\Tdens}(t)$ (the factor $2\Vect{\Length}^{-1}$ has no
  particular influence on the dynamics).  Due to the presence of the
  multiplicative term $\Vect{\Tdens}(t)$, the right-hand side
  of~\cref{eq:pp-ode-dyn} tends to zero on those entries where
  $\Vect{\Tdens}(t)$ tends to zero.  The numerical results presented
  in the above mentioned papers suggest that this scaling plays a
  crucial role in the good behavior of the numerical methods based on
  the dynamical system~\cref{eq:pp-ode}. However, using the
  transformation in~\cref{eq:trans-square} we can
  reinterpret~\cref{eq:pp-ode} as a ``classical'' gradient descent
  dynamics.  In fact, if we multiply
  component-wise~\cref{eq:gf-ode-dyn} by $\frac{1}{2}
  \Vect{\Gfvar}(t)$, we obtain
\begin{align*}
  \Dt{\Vect{\Cond}(t)}
  =
  \frac{1}{2}\Vect{\Gfvar}(t)
  \odot
  \Dt{\Vect{\Gfvar}(t)}
  =
  \frac{1}{2}
  \Vect{\Cond}(t)
  \odot 
  \left(
    (\Mgrad \Vect{\PotOp}\Of{\Trans(\Vect{\Gfvar})})^2-\Vect{1}
  \right)\; ,
\end{align*}
which is equal to~\cref{eq:pp-ode-dyn}.  Nevertheless,
in~\cref{eq:gf-ode-dyn}, we keep scaling the gradient direction
$\left( \Mgrad\Vect{\PotOp}\Of{\Vect{\Cond}(t)}\right)^2
-\Vect{1}$ by a multiplicative factor $\Vect{\Gfvar}(t)$
which tends to zero on the same entries of $\Vect{\Tdens}(t)$, since
$\Vect{\Gfvar}(t)=\sqrt{4\Vect{\Tdens}(t)}$.

\begin{remark}
  Other transformations, like 
  \begin{equation*}
    \Trans(\Vect{\Gfvar})=\Vect{\Gfvar}^{p} \quad p\geq 1 \quad \mbox{and}\quad  \Trans(\Vect{\Gfvar})=e^{\Vect{\Gfvar}}\; ,
  \end{equation*}
  have been considered in our experiments. However, in this paper we
  consider only the transformation in~\cref{eq:trans-square}, since
  it gives the best results in
  terms of global efficiency and robustness.
\end{remark}



\section{Numerical Solution}
\label{sec:numerical-solution}
\subsection{Time Discretization}
\label{sec:time-discretization}
In this section, we describe three approaches for the
time-discretization of~\cref{eq:gf-ode} and to get a steady-state
solution, i.e. we show how to build an approximate sequence
$(\Vect{\Gfvar}^{\tstep})_{\tstep=1,\ldots,\Kmax}$ such that
\begin{equation}
  \label{eq:steady-state}
  \Norm[\Vect{\Length}]{
    \Grad_{\Vect{\Gfvar}} \Ingf{\Lyap}(\Vect{\Gfvar}^{\tstep})
  }
  =
  \Norm[2]{
    \sqrt{\Vect{\Length}}
    \odot
    \Vect{\Gfvar}^{\tstep}
    \odot
    \frac{1}{2}
    (
    \ABS{\Mgrad \Vect{\Pot}^{\tstep}}^2-\Vect{1}
    )
  }
  \leq \TolOpt,
\end{equation}
where $\TolOpt>0$ is a fixed tolerance.  We denote with
$\MyOpt{\tstep}$ the step at which~\cref{eq:steady-state} is
satisfied, writing
$(\Vect{\MyOpt{\Gfvar}},\Vect{\MyOpt{\Pot}})=(\Vect{\Gfvar}^{\MyOpt{\tstep}},\Vect{\Pot}^{\MyOpt{\tstep}})$.
Our approximate solutions of~\cref{eq:beckmann,eq:dual,eq:min-lyap}
will be given, respectively, by
\begin{equation*}
  (\Vect{\MyOpt{\Dflux}}=  \Vect{\MyOpt{\Cond}} \odot \Mgrad \Vect{\MyOpt{\Pot}}, \Vect{\MyOpt{\Pot}},\Vect{\MyOpt{\Cond}})
  \quad
  \mbox{with}
  \quad
  \Vect{\MyOpt{\Cond}}=
  \Trans(\Vect{\MyOpt{\Gfvar}})
  \; .
\end{equation*}

\paragraph{Forward Euler/Gradient Descent}
The first approach is the forward Euler time-stepping, resulting in
the classical \emph{Gradient Descent} (GD) algorithm, where, given a
sequence $\Deltat[\tstep]>0$, the approximation sequence
$(\Vect{\Gfvar}^{\tstep})_{\tstep=1,\ldots,\Kmax}$ is given by the
following equation:
\begin{equation}
  \nonumber
  \label{eq:ee}
  \Vect{\Gfvar}^{\tstepp} = \Vect{\Gfvar}^{\tstep} + \Deltat[\tstep]
  \frac{\Vect{\Gfvar}^{\tstep}}{4} \odot \left( ( \Mgrad
  \Vect{\PotOp}\Of{\Trans(\Vect{\Gfvar}^{\tstep})})^2 -\Vect{1}
  \right),\quad
  \tstep=1,\ldots,\Kmax
  \; .
\end{equation}
At each time step, we only need to solve the linear system
\[
\Lapl[\Vect{\Tdens}^{\tstep}]\Vect{\Pot}^{\tstep} =\Vect{\Forcing}.
\]
A similar discretization scheme, applied to~\cref{eq:pp-ode}, has been
successfully adopted
in~\cite{Bonifaci:2019,Facca-et-al:2018,Facca-discrete:2018,
  Facca-et-al-numeric:2020, FACCA2020184}.\\

\paragraph{Accelerated Gradient Descent}
The second approach considered in this paper is the \emph{Accelerated
  Gradient Descent (AGD)}.  In the AGD approach, first introduced
in~\cite{Nesterov:1983}, the approximating sequence is generated as
follows:
\begin{align*}
  \Vect{\Gfvar}^{\tstepp}  &= \Vect{\Momentum}^{\tstep} -
  \Deltat[\tstep]
  \Grad_{\Vect{\Gfvar}}\Ingf{\Lyap}(\Vect{\Momentum}^{\tstep})
  \; ,
  \\
  \Vect{\Momentum}^{\tstep}
  &=
  \Vect{\Gfvar}^{\tstep} +
  \frac{\tstep-1}{\tstep+2}
  (\Vect{\Gfvar}^{\tstep}-\Vect{\Gfvar}^{\tstep-1})
  \; .
\end{align*}
Typically, the AGD performs better than GD, reaching steady-state with
fewer time steps (we refer to~\cite{Nesterov:2014} for a
complete overview on the AGD approach and on its applications).

In our problem, GD and AGD have practically the same computational
cost, requiring the solution of only a weighted Laplacian matrix per
time step.  However, both the GD and the AGD are explicit
time-stepping schemes as shown in~\cite{SRBD2017}, and thus, in both
approaches, the time step size $\Deltat[\tstep]$ cannot be taken too
large, due to well-known stability issues. Hence, a very large number
of time steps may be required to reach the optimum, affecting the
performance of the algorithm.  On the other hand, implicit schemes do
not suffer from this time step size limitation, at the cost of solving
a non-linear equation. We will see that, in our formulation, the
additional computational cost required at each time step is
offset by a drastic reduction of the total number of time steps
required to reach the steady-state solution.\\

\paragraph{Backward Euler-Gradient Flow} 
We present the backward Euler time stepping of~\cref{eq:gf-ode}, in
the form of finding a sequence $(\Vect{\Gfvar
}^{\tstep})_{\tstep=1,\ldots,\Kmax}$ that, given a sequence of
$\Deltat[\tstep]$ and fixing $\Vect{\Gfvar}^{0}=\Vect{\Gfvar }_0$,
solves the following minimization problem:
\begin{align}
  \label{eq:gf-inf}
  \Vect{\Gfvar}^{\tstepp}&=
  \Argmin_{\Vect{\Gfvar}\in \REAL^{\Nedge} } 
  \left\{
  \Ingf{\Lyap}(\Vect{\Gfvar})
  +
  \frac{1}{2\Deltat[\tstep]}
  \Norm[\Vect{\Length}]{\Vect{\Gfvar}-\Vect{\Gfvar}^{\tstep}}^2
  \right\}
  \; .
\end{align}
Readers familiar with optimization techniques will recognize
in~\cref{eq:gf-inf} the \emph{Proximal Mapping}, while readers
familiar with the discretization of ODEs will recognize the classical
backward-Euler discretization by casting the Euler-Lagrange Equations
of~\cref{eq:gf-inf} in the form
\begin{align*}
  \frac{ 
    ( \Vect{\Gfvar}^{\tstepp} - \Vect{\Gfvar}^{\tstep} )
  }{
    \Deltat[\tstep]
  }
  &=
  -
  \Mlength^{-1}\Differential_{\Vect{\Gfvar}} \Ingf{\Lyap}(\Vect{\Gfvar}^{\tstepp})
  =
  \frac{1}{4}
  \Vect{\Gfvar}^{\tstepp} 
  \odot 
  \left(
    (\Mgrad \Vect{\PotOp}\Of{\Trans(\Vect{\Gfvar}^{\tstepp})})^2-\Vect{1}
  \right)\;.
\end{align*}
However, it is better to rewrite~\cref{eq:gf-inf} in the form of an
inf-sup problem, using the definition of $\Lyap$
in~\cref{eq:min-max,eq:joule}. In fact, the Lagrangian functional
$\Lagr$ in~\cref{eq:min-max} in the $\Vect{\GfCond}$-variable becomes
\begin{equation}  
  \Ingf{\Lagr}(\Vect{\Volt},\Vect{\GfCond})
  =
  \Lagr(\Vect{\Volt},\Trans(\Vect{\GfCond}))
  =
  \Vect{\Forcing}^T \Vect{\Volt}  
  - 
  \frac{1}{2}\Vect{\Volt}^T \Lapl[\Trans(\Vect{\GfCond})] \Vect{\Volt} 
  + 
  \frac{1}{2}  \Trans(\Vect{\GfCond})^T \Vect{\Length}\;,
\end{equation}
hence we look for the solution of the following problem:
\begin{gather}
  \label{eq:inf-sup-gradient-flow}
  (\Vect{\Volt}^{\tstepp},\Vect{\Gfvar}^{\tstepp})=
  \Argmax_{\Vect{\Volt}\in\REAL^{\Nnode}}
  \Argmin_{\Vect{\Gfvar}\in \REAL^{\Nedge} } 
  \left\{
    \Ingf{\InfSupGf}(\Vect{\Volt},\Vect{\Dgfvar},\Vect{\Dgfvar}^\tstep,\Deltat[\tstep])
    :=
    \Ingf{\Lagr}(\Vect{\Volt},\Vect{\GfCond})
    +
    \frac{1}{2\Deltat[\tstep]} \Norm[2,\Vect{\Length}]{\Vect{\Gfvar}-\Vect{\Gfvar}^{\tstep}}^2
  \right\}\;.
\end{gather}
The solution pair
$(\Vect{\Volt}^{\tstepp},\Vect{\Gfvar}^{\tstepp})$ can be
found as the stationary point of the functional
$\Ingf{\InfSupGf}$ in
\cref{eq:inf-sup-gradient-flow}, described by the following
non-linear system of equations:
\begin{equation}
  \label{eq-non-linear-gfvar}
  0\!=\!
  \begin{pmatrix} 
    -
    \Differential _{\Vect{\Pot}}
    \Ingf{\InfSupGf}
    \\
    -
    \Differential _{\Vect{\Gfvar}}
    \Ingf{\InfSupGf}
  \end{pmatrix}
  \!=\!
  \begin{pmatrix}
    \Vect{\Fnewtonpot}(\Vect{\Volt},\Vect{\Dgfvar},\Vect{\Dgfvar}^\tstep,\Deltat[\tstep])
    \\ 
    \Vect{\Fnewtongfvar}(\Vect{\Volt},\Vect{\Dgfvar},\Vect{\Dgfvar}^\tstep,\Deltat[\tstep])
  \end{pmatrix}
  =
  \begin{pmatrix}
    \Lapl[\Trans(\Vect{\GfCond})]\Vect{\Volt}-\Vect{\Forcing}
    \\
    \Vect{\Length}
    \!\odot\!
    \left(
      \frac{1}{4}\Vect{\GfCond}
      \odot
      \left( (\Mgrad\Vect{\Volt})^{2} - \Vect{1} \right)
      -
      \frac{1}{ \Deltat[\tstep]}
      \left(\Vect{\Dgfvar}-\Vect{\Dgfvar}^{\tstep}\right)
    \right)
  \end{pmatrix}.
\end{equation}
At each time step $\tstep$, we apply a damped version of the inexact
Newton-Raphson method to~\cref{eq-non-linear-gfvar} that, given
$0<\Damp\leq 1$, amounts to finding the sequence
$(\Vect{\Volt}^{\tstep,\Newi},\Vect{\Dgfvar}^{\tstep,\Newi})_{\Newi=1,\ldots,\Rmax}$ with
$(\Vect{\Volt}^{\tstep,1},\Vect{\Dgfvar}^{\tstep,1})=(\Vect{\Volt}^{\tstep},\Vect{\Dgfvar}^{\tstep})$
and defined as follows:
\begin{align}
  \label{eq:newton}  
   \Matr{\Jacgf}\Of{\Vect{\Volt}^{\tstep,\Newi},\Vect{\Dgfvar}^{\tstep,\Newi},\Deltat[\tstep]}
   \begin{pmatrix}
     \Vect{x}
     \\
     \Vect{y}
   \end{pmatrix}
   &= 
   -
   \begin{pmatrix}
     \Vect{\Fnewtonpot}(\Vect{\Volt}^{\tstep,\Newi},\Vect{\Dgfvar}^{\tstep,\Newi},\Vect{\Dgfvar}^\tstep,\Deltat[\tstep])
     \\ 
     \Vect{\Fnewtongfvar}(\Vect{\Volt}^{\tstep,\Newi},\Vect{\Dgfvar}^{\tstep,\Newi},\Vect{\Dgfvar}^\tstep,\Deltat[\tstep])
   \end{pmatrix}\;,
   \\
   \label{eq:newton-step}
   \begin{pmatrix} 
     \Vect{\Volt}^{\tstep,\Newii}\\
     \Vect{\Gfvar}^{\tstep,\Newii}
   \end{pmatrix}
   &=
   \begin{pmatrix} 
     \Vect{\Volt}^{\tstep,\Newi}\\
     \Vect{\Gfvar}^{\tstep,\Newi}
   \end{pmatrix}
   +
   \Damp
   \begin{pmatrix}
     \Vect{x}
     \\
     \Vect{y}
   \end{pmatrix}
   \;,
\end{align}
where matrix $\Matr{\Jacgf}$ is the Jacobian matrix of the function
$(\Vect{\Fnewtonpot},\Vect{\Fnewtongfvar})$. At each Newton step we
find an approximate solution of the linear system in~\cref{eq:newton}
such that
\begin{equation}
  \label{eq:relative-res}
  \Norm[2]{
  \Matr{\Jacgf}
  \begin{pmatrix}
    \Vect{x}
    \\
    \Vect{y}
  \end{pmatrix}
  +
  \begin{pmatrix}
    \Vect{\Fnewtonpot}
    \\ 
    \Vect{\Fnewtongfvar}
  \end{pmatrix}
  }
  \leq
  \TolLinearNewton
  \Norm[2]{
    \begin{pmatrix}
      \Vect{\Fnewtonpot}
      \\ 
      \Vect{\Fnewtongfvar}
    \end{pmatrix}
  }\;,
\end{equation}
with $\TolLinearNewton>0$, iterated until
\begin{equation}
  \label{eq:non-linear}
  \Norm[2]{
    \begin{pmatrix}
      \Vect{\Fnewtonpot}/\Norm{\Forcing}
      \\ 
      \Vect{\Fnewtongfvar}
    \end{pmatrix}
  }
  \leq
  \TolNonLinear\;,
\end{equation}
with $\TolNonLinear>0$ (we scale by $\Norm{\Forcing}$ in order to get
a relative residual in the solution of the linear system that defines
the function $\Vect{\Fnewtonpot}$).  The damping parameter should be
kept ideally $\Damp=1$ in order to get the quadratic convergence of
the Newton method. However, smaller $\Damp$ is used in order to
prevent failures of the linear and the non-linear solvers.

The Jacobian matrix $\Matr{\Jacgf}$ in~\cref{eq:newton} can be
calculated as follows:
\begin{align}
  \label{eq:jacobian-gf}
  \Matr{\Jacgf}\Of{\Vect{\Volt},\Vect{\Dgfvar},\Deltat} 
  &=
  \begin{pmatrix}
    \Matr{\Ablock}\Of{\Vect{\Dgfvar}}  & \Matr{\Bblock}^T\Of{\Vect{\Volt},\Vect{\Dgfvar}}
    \\
    \Matr{\Bblock}\Of{\Vect{\Volt},\Vect{\Dgfvar}} & - \Matr{\Cblock}\Of{\Deltat,\Vect{\Volt},\Vect{\Dgfvar}} 
  \end{pmatrix}
  \; ,
\end{align}
where
\begin{align*}
  \Matr{\Ablock}\Of{\Vect{\Dgfvar}} 
  &= 
  \Lapl[\Trans(\Vect{\GfCond})]=
  \Matr{\Diff}\Diag{\Trans(\Vect{\GfCond})}\Mgrad;\,
  \\
  \Matr{\Bblock}\Of{\Vect{\Volt},\Vect{\Dgfvar}}
  &=
  \frac{1}{2} \Diag{ (\Mgrad\Vect{\Volt})\odot \Vect{\GfCond}} \Matr{\Diff}^T\;,
  \\
  \Matr{\Cblock}\Of{\Deltat,\Vect{\Volt},\Vect{\Dgfvar}}
  &=
  \Diag{
    \Vect{\Length}
    \odot
    \left(    
      \frac{1}{4}
      \odot 
      \left(
        (\Mgrad\Vect{\Volt})^{2} - \Vect{1}
      \right)
      +
      \frac{1}{\Deltat}  
    \right)
  }
  \; .
\end{align*}
We will refer to the procedure described in this section as the
\emph{Gradient Flow} (GF) approach, as in \cite{ambrosio2008gradient}.


\begin{remark}
  Here, it is worth highlighting some similarities between the
  Gradient Flow approach proposed in this paper and the interior point
  methods for the solution of minimization problems with non-linear
  constraints~\cite{Wright:1997}. In fact, the functional $\Lyap$ is
  nothing but the Lagrangian functional of the dual problem
  of~\cref{eq:dual}, where the constraints are rewritten as in
  \cref{eq:dual-original,eq:dual-transformed}.  In interior point
  methods, this Lagrangian functional is relaxed with a penalization
  functional of the constraints (typically a logarithmic barrier
  function), scaled by a relaxation factor $\nu>0$, that is
  progressively reduced along the optimization process.  In our
  approach the relaxation functional is the distance square term
  scaled by the factor $\frac{1}{2\Deltat}$, which plays the role of
  the factor $\nu$. In both schemes, the stationary equations for the
  scaled functional are solved via inexact Newton. Moreover the
  sequence of approximated solutions is following in both cases an
  underlying ``curve'', the integral solution of the
  ODE~(\ref{eq:gf-ode}) for the gradient descent equation, and the
  Central Path of the interior point methods.
\end{remark}

\subsection{Linear system solution}
\label{sec-precs}
All the discretization schemes described
in~\cref{sec:time-discretization} require the solution of a sequence
of sparse linear systems. In the GD and the AGD approaches, these
linear systems only involve weighted Laplacian matrices, thus
algebraic multigrid solvers can efficiently handle these problems,
often with computational time that scales linearly with the size of
the matrix. In particular, in our experiments we adopt the
AGgregation-based algebraic MultiGrid (AGMG) software described
in~\cite{Notay:2012} and the Lean Algebraic MultiGrid (LAMG) described
in \cite{Livne-Brandt:2012}, which has been developed specifically for
the solution of linear systems involving weighted Laplacian matrices.
The GF approach instead requires the solution of the saddle point
linear system in~\cref{eq:relative-res}.  We refer
to~\cite{begoli2005} for a complete overview on this topic. We propose
two approaches to tackle this linear algebra problem.\\

\paragraph{Reduced-MG approach}
The first approach proposed in this paper
exploits the fact that the matrix $\Matr{\Cblock}$ is diagonal, thus
it can be trivially inverted, under some restriction on the time step
size $\Deltat$.  The resulting reduced linear system reads as
\begin {subequations}
  \nonumber
  \label{eq:reduced} 
  \begin{align}
    \label{eq:reduced-volt} 
    \overbrace{\left(\Matr{\Ablock} + \Matr{\Bblock}^T 
        \Matr{\Cblock}^{-1}
        \Matr{\Bblock} \right) }^{\SchurMC}
    \Vect{x} &
    = 
    \Vect{\RhsVolt} +  \Matr{\Bblock}^T\Vect{\RhsGfvar}= \Vect{\bar{\RhsVolt}}
    \; ,
    \\
    \label{eq:reduced-gfvar} 
    \Vect{y} &= 
    \Matr{\Cblock}^{-1}
    \left(
      \Matr{\Bblock}\Vect{x}  -     \Vect{\RhsGfvar}
    \right) \; ,
  \end{align}
\end{subequations}
where matrix $\SchurMC$ is nothing but the Schur complement of the
block $\Matr{\Cblock}$ of the matrix $\Matr{\Jacgf}$, which can be
rewritten in the form
\begin{equation}
  \nonumber
  \label{eq:jacobian-reduced}
  \SchurMC
  =
  \Matr{E}\Diag{\Vect{\bar{\Tdens}}}\Mlength^{-1}\Matr{E}^T
  =
  \Lapl[\Vect{\bar{\Tdens}}]\; ,
\end{equation}
with $\Vect{\bar{\Tdens}}$ given by
\begin{equation*}
  \Vect{\bar{\Tdens}}=\Vect{\Tdens} +
  \Deltat
  \frac{
    \left(\frac{1}{2} \Vect{\Gfvar}\odot (\Mgrad\Vect{\Volt} )\right)^{2}
  }{
    \Vect{1} - 
    \frac{\Deltat }{4}((\Mgrad\Vect{\Volt} )^2-\Vect{1})
  }\;.
\end{equation*} 
As long as $\Vect{\bar{\Tdens}}$ is positive, matrix $\Matr{\Cblock}$
takes the form of a weighted Laplacian matrix, thus it can be solved
efficiently via algebraic multigrid solvers. We denote this approach
with \emph{Reduced-MG}. The downside of the Reduced-MG approach is
that we are forced to impose a restriction on the time step size
$\Deltat[\tstep]$, thus more time steps are required to reach steady
state.  In fact, in order to preserve positivity of
$\Vect{\bar{\Tdens}}$ and to avoid division by numbers close to zero,
at each time step, before starting the Newton process, we choose
$\Deltat[\tstep]$ to satisfy the constraint
\begin{equation}
  \label{eq:tol-C}
  \max_{\Iedge=1,\ldots,\Nedge}
  \left\{
  \MatrPlus[\Iedge]{\Iedge}{\Cblock}
  =
  \frac{1}{\Deltat[\tstep]}-(\Mgrad\Vect{\Volt})_{\Iedge}^{2} + 1
  \right\}
  \geq \TolC >0. 
\end{equation}
Then, the damping parameter $\Damp$ in~\cref{eq:newton-step} is tuned
so that~\cref{eq:tol-C} holds at each step of the Newton process. We
start from $\Damp=1$ (which corresponds to the classical Newton
method) and we progressively reduce $\Damp$ until
condition~\ref{eq:tol-C} is met.  If $\Damp$ becomes smaller than a
prefixed lower bound $\MinDamp$ (which in our experiments is fixed at
$\MinDamp=5\times10^{-2}$) we abort the Newton procedure, restarting it after
halving the time step size $\Deltat[\tstep]$. Although having
$\Damp<1$ destroys the quadratic convergence of the Newton-Raphson
approach, our numerical experiments suggests that it prevents in most
cases the failure of the Newton method. After a few damped Newton
iterations, $\Damp$ remains equal to $1$ and the quadratic
convergence is restored.\\

\paragraph{Full-LSC approach}
A possible strategy for avoiding the time step size limitation of the
Reduced-MG approach is to solve the entire saddle point system instead
of the reduced system.  However, we want to avoid
forming and then approximating the inverse of the Schur complement
of the block $\Matr{\Ablock}$ of the matrix $\Matr{\Jacgf}$, which
is given by
\begin{equation*}
  \Matr{\SchurMA}=
  -(\Matr{\Cblock}+\Matr{\Bblock}\Matr{\Ablock}^{-1}\Matr{\Bblock}^T)\; .
\end{equation*}
To this aim, we first note that the matrix
$\Matr{\Jacgf}$ in \cref{eq:jacobian-gf} can be factored as follows
\begin{align}
  \label{eq:W-least-square}
  \Matr{\Jacgf}&=
  \Matr{\Zblock}
  \Matr{\Wblock}\Of{\Deltat,\Vect{\Pot},\Vect{\GfCond}}
  \Matr{\Zblock}^T \; ,
\end{align}
where matrices $\Matr{\Wblock}$ and $\Matr{\Zblock}$
are given by
\begin{align*}
  \Matr{\Wblock}\Of{\Deltat,\Vect{\Pot},\Vect{\GfCond}}
  &=
  \begin{pmatrix}
    \Diag{\Trans(\Vect{\GfCond})}
    &
    \frac{1}{2} \Diag{\Vect{\GfCond}\odot (\Mgrad \Vect{\Pot})}
    \\
    \frac{1}{2} \Diag{\Vect{\GfCond} \odot  (\Mgrad \Vect{\Pot})}
    &
    -
    \Diag{
      \frac{1}{\Deltat}
      -
      \frac{1}{4}
      \odot
      \left(
      (\Mgrad \Vect{\Pot})^2 -\Vect{1} 
      \right)
    }
    \\
  \end{pmatrix}
  \; ,
  \\
  \Matr{\Zblock}
  &=
  \begin{pmatrix}
    \Matr{\Diff} \Mlength^{-1/2}
    &
    0
    \\
    0
    &
    \Mlength^{1/2}
    \\
  \end{pmatrix}\; .
\end{align*}
The factorization of $\Matr{\Jacgf}$ in~\cref{eq:W-least-square} and
the fact that matrices $\Matr{\Zblock},\Matr{\Zblock}^T$ resembles the
divergence and gradient operators in the graph setting, make us
consider the \emph{Least Square Commutator} (LSC) preconditioner,
introduced in~\cite{Elman2006}, as a preconditioner for solving the
saddle point linear systems in our problem.  The LSC preconditioner,
applied to the solution of the linear system in~\cref{eq:newton},
can be written as
\begin{align*}
  \Matr{\PrecW}
  &=
  \Pinv{(\Matr{\Zblock} \Matr{\Zblock}^T)}
  \Matr{\Zblock}
  \Matr{\Wblock}^{-1}
  \Matr{\Zblock}^T
  \Pinv{(\Matr{\Zblock} \Matr{\Zblock}^T)}
  \; ,
  \\
  &=
  \begin{pmatrix}
    \Pinv{\Lapl}[\Vect{1}]
    &
    \!0
    \\
    \!0
    &
    \!\Mlength^{-1}
    \\
  \end{pmatrix}
  \Matr{\Zblock}
  \Matr{\Wblock}^{-1}
  \Matr{\Zblock}^T
  \begin{pmatrix}
    \Pinv{\Lapl}[\Vect{1}]
    &
    0
    \\
    0
    &
    \Mlength^{-1}
    \\
  \end{pmatrix}\; .
\end{align*}
We use the preconditioner $\Matr{\PrecW}$ within the flexible GMRES
algorithm~\cite{Saad1993}.  This preconditioner would be extremely
appealing from the computational point of view, since, besides a few
diagonal scalings and sparse matrix-vector products, it only relies on
solving twice a linear system with an unweighted Laplacian
matrix. Unfortunately, the performance of the preconditioner degrades
after a few time steps starting form $\Vect{\Tdens}_0=\Vect{1}$.  This
suggests considering a second factorization of the Jacobian matrix,
namely
\begin{equation}
  \nonumber
  \label{eq:V-least-square}
  \Matr{\Jacgf}
  =
  \Matr{\ZTDblock}
  \Matr{\Vblock}\Of{\Deltat,\Vect{\Pot},\Vect{\GfCond}}
  \Matr{\ZTDblock}^T,
\end{equation}
where matrices  $\Matr{\Vblock}$ and  $\Matr{\ZTDblock}$ 
are given by
\begin{align*}
  \Matr{\Vblock}\Of{\Deltat,\Vect{\Pot},\Vect{\GfCond}}
  &=
  \begin{pmatrix}
    \Matr{I}
    &
    \Diag{\Mgrad \Vect{\Pot}}
    \\
    \Diag{\Mgrad \Vect{\Pot}}
    &
    -
    \Diag{
      \frac{1}{\Deltat}
      -
      \frac{1}{4}
      \odot
      \left(
        (\Mgrad \Vect{\Pot})^2 -\Vect{1} 
      \right)
    }
    \\
  \end{pmatrix}
  \; ,
  \\
   \label{eq:z-matrices}
  \Matr{\ZTDblock}
  &=
  \begin{pmatrix}
    \Matr{\Diff} \Mlength^{-1/2} \Diag{\Vect{\Tdens} }^{1/2}
    &
    0
    \\
    0
    &
    \Mlength^{1/2}
    \\
  \end{pmatrix}
  \; .
\end{align*}
Note that this factorization strongly relies on the transformation
chosen, $\Trans(\Vect{\Gfvar})=(\Vect{\Gfvar})^2/4$.  The resulting
LSC preconditioner reads as follows:
\begin{align*}
  \Matr{\PrecV}
  &= \Pinv{( \Matr{\ZTDblock} \Matr{\ZTDblock}^T)}
  \Matr{\ZTDblock}
  \Matr{\Vblock}^{-1}
  \Matr{\ZTDblock}^T
  \Pinv{( \Matr{\ZTDblock} \Matr{\ZTDblock}^T)}
  \; ,
  \\
  &=
  \begin{pmatrix}
    \Pinv{\Lapl[\Vect{\Tdens}]}
    &
    0
    \\
    0
    &
    \Mlength^{-1}
    \\
  \end{pmatrix}
  \Matr{\ZTDblock}
  \Matr{\Vblock}^{-1}
  \Matr{\ZTDblock}^T
   \begin{pmatrix}
     \Pinv{\Lapl[\Vect{\Tdens}]}
    &
    0
    \\
    0
    &
    \Mlength^{-1}
    \\
   \end{pmatrix}
   \; ,
\end{align*}
again used within the flexible GMRES algorithm.
Similar to the reduced system approach, at each time step and at each
Newton step we restrict the time step size $\Deltat[\tstep]$
and the damping parameter $\Damp$ so that the following condition
holds:
\begin{equation}
  \label{eq:tol-detsV}
  \max\left\{
    \frac{1}{\Deltat}
    -
    \frac{1}{4}
    \left(
    (\Mgrad \Vect{\Pot})^2 -\Vect{1} 
    \right)
    +
    (\Mgrad \Vect{\Pot})^2
    \right\}
  \geq
  \TolC
  \; ,
\end{equation}
in order to avoid inversion of matrices close to being singular.  In
this case, the main computational effort is spent in the solution of
two linear systems with $\Vect{\Tdens}$-weighted Laplacian matrices,
which can be achieved with the algebraic
multigrid approach. We denote this preconditioner approach as
\emph{full-LSC}.

\section{Numerical Experiments}
\label{sec:numerical-experiments}
In this section, we describe three test cases we devised in order to
test the accuracy and efficiency of the different methods proposed for
the solution of the \OTP\ on graphs.\\

\paragraph{Test-case~1}
The first test-case problem is the Single Source Shortest Path (SSSP),
which can be described as follows. Fix a node $\bar{\Node}$ in
the graph $\Graph$, say $\bar{\Node}=\Node_1$. We want to find the
shortest-path distance from the node $\Node_1$ to all other nodes in
$\Graph$. Thus we want to determine
\begin{equation}
  \nonumber
  \label{eq:sssp}
  \Vect{\Opt{d}}\in \REAL^{\Nnode}\;,
  \quad  
  \Vect[\Inode]{\Opt{d}} :=\Dist_{\Graph}(\Node_{1},\Node_{i})
  \; .
\end{equation}

The solution of the SSSP is naturally related to the problem of
finding a shortest path tree rooted at $\Node_1$.  This is a spanning
tree $\Tree$ of $\Graph$ such that, given any node $\Node$ in
$\Nodes\setminus\Node_1$, the path from the root $\Node_1$ to $\Node$
(unique since $\Tree$ is a tree) is a shortest path from $\Node_1$ to
$\Node$ in $\Graph$. Note that there may exist multiple shortest path
trees.  The SSSP and the OTP on graphs are related by the following
proposition.
\begin{Prop}
  The solution of the SSSP problem rooted at $\Node_1$ is equivalent
  to solving the \OTP\ on graphs with the following source and sink
  vectors:
  \begin{equation}
    \nonumber
    \label{eq:forcing-eikonal}
    \Vect[\Inode]{\Source}=
    \left\{
    \begin{aligned}
      & 0 &\mbox{if } \Inode=1\\
      & 1/(\Nnode-1) &\mbox{otherwise}
    \end{aligned} 
    \right.
    \; ,
    \quad
    \Vect[\Inode]{\Sink}=
    \left\{
    \begin{aligned}
      & 1 &\mbox{if } \Inode=1\\
      & 0 &\mbox{otherwise}
    \end{aligned} 
    \right.
    \; .
  \end{equation}
  Given any shortest-path tree $\Tree$ rooted at $\Node_1$, the
  vectors
  \begin{align}
    \label{eq:eik-optpot}
    \Vect{\OptPot} &
    =  \Vect{\Opt{d}}
    \; ,
    \\
    \label{eq:eik-opttdens}
    \Vect[\Iedge]{\Opt{\Tdens}}(\Tree) 
    &:= 
    \left\{
    \begin{aligned}
      &
      \frac{1}{\Nnode-1}
      \left\{
      \parbox{7cm}{Number of shortest paths in $\Tree$ between
        any node
        $\Node_\Inode$ and $\Node_1$ passing through edge $\Iedge$}
      \right\} &\mbox{if } \Iedge \in \Tree \\ &0 &\mbox{otherwise}
    \end{aligned}
    \right.
    \; ,
    \\
    \label{eq:eik-optvel}
    \Vect{\Opt{\Dflux}} &= \Vect{\Opt{\Tdens}}\odot \Mgrad \Vect{\OptPot}
    \; ,
  \end{align}
  solve \cref{eq:dual,eq:min-lyap,eq:beckmann} respectively.
\end{Prop}
\begin{proof}
The proof of the above Proposition is based on the fact that vectors
$\Vect{\Opt{\Dflux}}$ and $\Vect{\Opt{\Pot}}$
in~\cref{eq:eik-optvel,eq:eik-optpot} are admissible solutions of
\cref{eq:beckmann,eq:dual}, and the duality gap
in~\cref{eq:duality-gap} is zero, ensuring optimality of both
solutions.  Note that this holds true taking any convex combination of
elements in the set $\left\{ \Vect{\Opt{\Tdens}}(\Tree) \ : \:
\Tree=\mbox{Shortest path Tree rooted at }\Node_1\right\}$.
\end{proof}

This test case is particularly relevant in our discussion, being an
example where the solution of~\cref{eq:dual} is unique (up to a
constant), while the solutions of~\cref{eq:beckmann,eq:min-lyap} are
not.  In fact, for any $\Vect{\Opt{\Cond}}$, the kernel of the matrix
$\Lapl[\Vect{\Opt{\Cond}}]$ consists only of the constant vectors,
since each node of the graph is ``surrounded'' by at least one edge
with non-zero conductivity. In contrast, the attractor of the
dynamical equation~\cref{eq:gf-ode} is the image through the map
$\Trans^{-1}$ of any convex combination of $\left\{
\Vect{\Opt{\Tdens}}(\Tree) \, : \, \Tree=\mbox{Shortest path
  Tree}\right\}$.

In this test case we consider graphs describing the node connection of
a regular grid triangulation of the square $[0,1]\times[0,1]$, thus two
coordinates $(x_{\Inode},y_{\Inode})$ are assigned to each node
$\Inode$.  The vector $\Vect{\Length}$ is the Euclidean distance
between the nodes connected by two edges.  Node $\Node_1$ has
coordinates $(0.5,0)$.  We generate a sequence of graphs
$(\Graph_{\Igraph}=(\Nodes_{\Igraph},\Edges_{\Igraph}))_{\Igraph=0,\ldots,5}$
starting from an initial triangulation $\Triang_0$ (reported
in~\cref{fig:rect-opt}), conformally refined 5 times.\\

\paragraph{Test-case~2}
The second test-case problem is the graph counterpart of
Test-case~1 considered in \cite{Facca-et-al-numeric:2020}, where the
OTP is posed in the continuous setting $\Banach=[0,1]\times[0,1]$.
We consider the same sequence of graphs
$(\Graph_{\Igraph})_{\Igraph=0,\ldots,5}$ described in Test-case~1 and a
forcing term $\Vect{\Forcing}$ given by
\begin{equation}
  \nonumber
  \label{eq:forcing-tc1}
  \Vect[\Inode]{\Forcing}
  =
  \left\{
    \begin{aligned}
      &C &\mbox{ if } x_{\Inode},y_{\Inode}\in [0.125,0.375]\times[0.25,0.75]\\
      &-C &\mbox{ if } x_{\Inode},y_{\Inode}\in [0.625,0.875]\times[0.25,0.75]\\
      &0 &\mbox{ otherwise} 
    \end{aligned}
  \right. \; ,
\end{equation}
where $C=(\sqrt{\Nnode}-1)$.  For this test-case there exist
explicit solutions for~\cref{eq:beckmann,eq:dual,eq:min-lyap}.  In
fact, taking the barycenters of the edges of $\Graph$
\begin{equation*}
  (\bar{x},\bar{y})_\Iedge
  =
  0.5(x_{\Inode}+x_{\Jnode},y_{\Inode}+y_{\Jnode})
  \; ,\; e=(\Node_{\Inode},\Node_{\Jnode}),
\end{equation*}
a solution $\Vect{\OptTdens}$ of~\cref{eq:min-lyap} is given by
\begin{equation}
  \nonumber
  \label{eq:rect-optdens}
  \Vect[\Iedge]{\OptTdens}
  =
  \left\{
    \begin{aligned}
      &C(\bar{x}_{\Iedge}-0.125) & &\mbox{ if } \bar{x}_{\Iedge}\in
      [0.125,0.375],& &y_{\Node_1}=y_{\Node_2} \in[0.25,0.75]
      \\
      &2
      & &\mbox{ if } \bar{x}_{\Iedge}\in [0.375,0.625],
      & &y_{\Node_1}=y_{\Node_2}
      \in[0.25,0.75]
      \\
      &C(0.875-\bar{x}_{\Iedge})
      & &\mbox{ if }
      \bar{x}_{\Iedge}\in [0.625,0.875],& &y_{\Node_1}=y_{\Node_2}
      \in[0.25,0.75]
      \\
      &0 & &\mbox{ otherwise}
    \end{aligned}
  \right.,
\end{equation}
while a solution of \cref{eq:dual} is given by
\begin{equation}
  \label{eq:rect-optpot}
  \Vect[\Inode]{\OptPot}
  =
  x_{\Inode}
  \; .
\end{equation}
Then, it is easy to check that $
\Vect{\Opt{\Dflux}}=\Vect{\Opt{\Tdens}}\odot \Mgrad \Vect{\OptPot} $
and $\Vect{\OptPot}$ satisfy the constraints of
\cref{eq:beckmann,eq:dual} with zero duality gap, proving that they
are optimal solutions.  Note that there exist infinitely many
solutions of~\cref{eq:dual}. In fact, it is possible to ``perturb''
the values of the solution $\Vect{\OptPot}$ in~\cref{eq:rect-optpot}
on those nodes ``outside'' the rectangle
$[0.125,0.875]\times[0.25,0.75]$, still satisfying the constraints
of~\cref{eq:dual}, while the value of the cost functional
$\Vect{\Forcing}^T\Vect{\OptPot}$ remains the same.

\begin{figure}
  \centerline{
    \includegraphics[width=0.49\textwidth,
    trim={0.0cm 0cm 0.0cm 0.0cm},clip
    ]{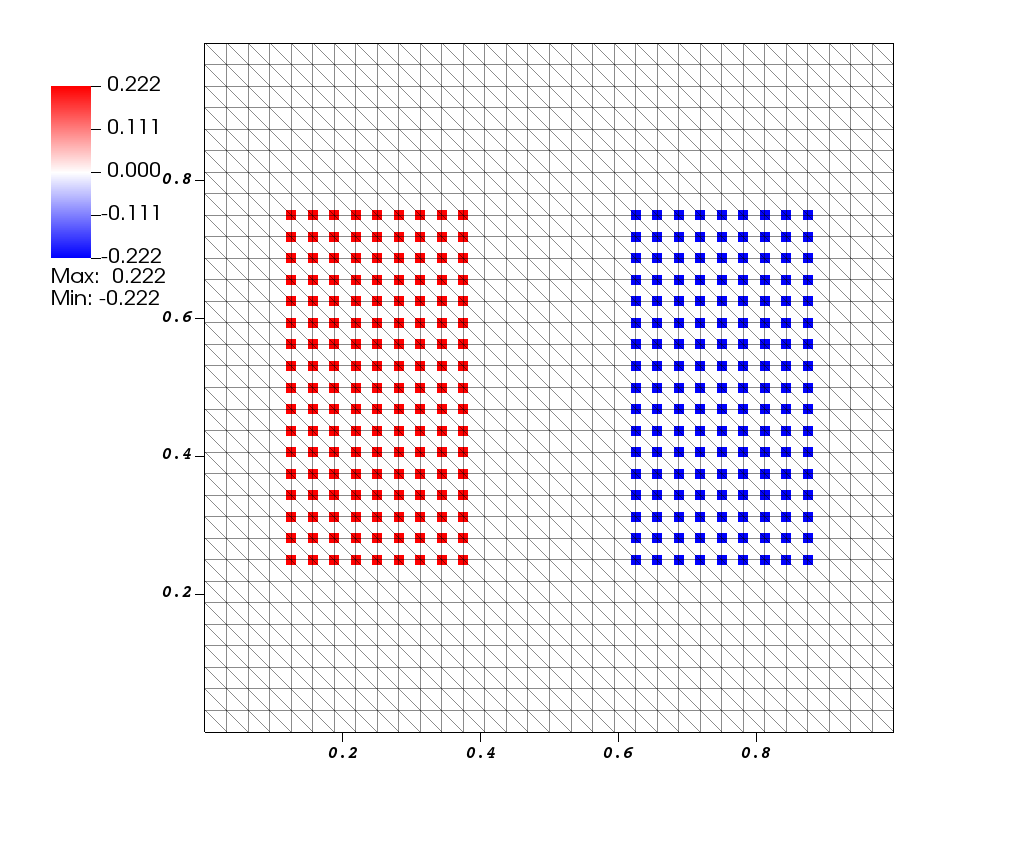}
    \includegraphics[width=0.49\textwidth,
    trim={0.0cm 0cm 0.0cm 0.0cm},clip
    ]{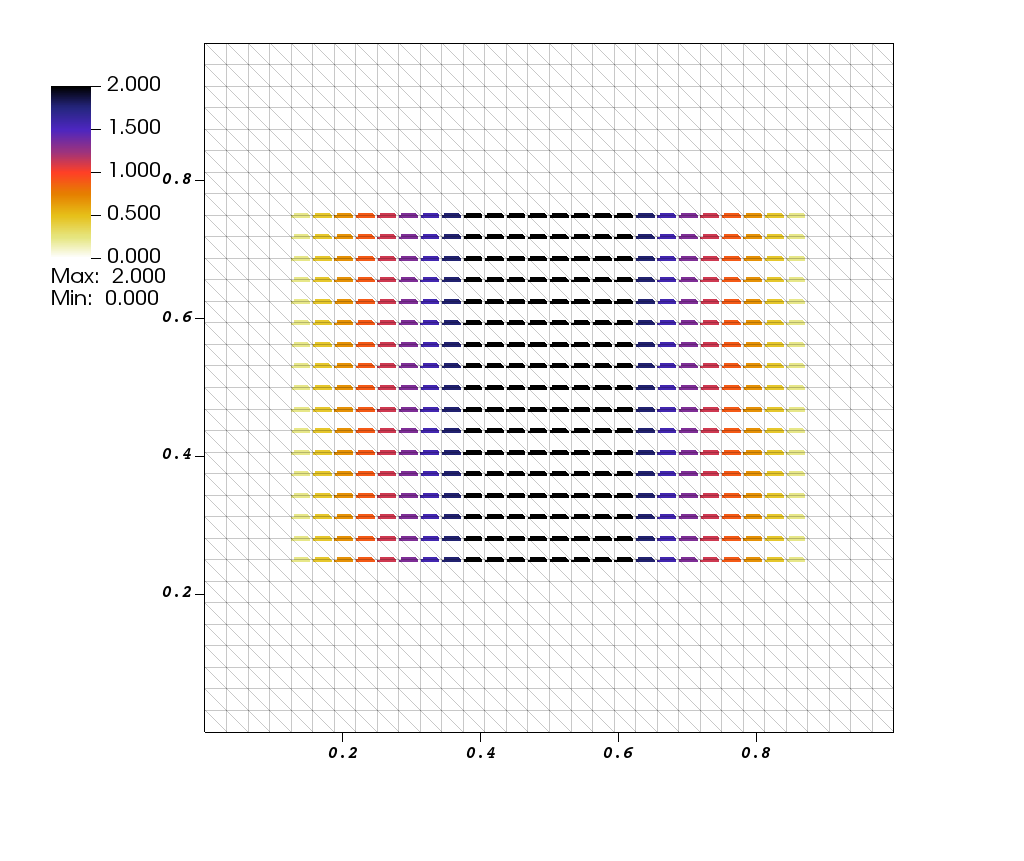}
  }
  \caption{
    We report on the left the spatial distribution of the forcing 
    term $\Vect{\Forcing}$ for the Test-case~2, for the graph
    $\Graph_0$. On the right panel we report the corresponding 
    solution $\Vect{\OptTdens}$ of~\cref{eq:min-lyap}.
  }
  \label{fig:rect-opt}
\end{figure} 
In~\cref{fig:rect-opt}, we report the graph $\Graph_0$, together with
the spatial distribution of the forcing term $\Vect{\Forcing}$ and
optimal solution $\Vect{\OptTdens}$.\\

\paragraph{Test-case~3}
In the third test-case, we consider graphs less structured than those
in Test-case~1 and Test-case~2. In particular we consider three
well-known categories of synthetic random graphs: the
Erd\"os-R\'enyi~\cite{Erdos:1959}, the
Watts-Strogatz~\cite{Watts1998}, and the Barabasi-Albert graphs
\cite{Barabasi99}. We generate graphs with increasing number of
nodes/edges using the Python package Networkx (see
\cite{SciPyProceedings_11}).  We repeat this procedure 10 times,
getting groups of 10 graphs with similar characteristics (number of
nodes, number of edges, construction procedure) but different topology
structure.  We assign to each graph a different forcing term
$\Vect{\Forcing}$ and weight $\Vect{\Length}$.  The forcing term
vector $\Vect{\Forcing}$ is generated randomly within a uniform
distribution in $[-1,1]$ on a fraction of the nodes. We consider two
scenarios, one where all entries of the forcing term $\Vect{\Forcing}$
are non-zeros (similar to Test-case~1) and another where this holds
only for the $10\%$ of the entries (similar to Test-case~2). The
negative entries of the resulting vector $\Vect{\Forcing}$ are then
scaled so that $\Vect{\Forcing}^T\Vect{1}=0$. The weight vector
$\Vect{\Length}$ is generated randomly with a uniform distribution in
$[0.5,1.5]$.

This procedure is done in order to average the metrics of our
simulations (total number of time steps, Newton steps, CPU time,
etc.) when the results are presented.

\section{Numerical Results}
\label{sec:gd-agd-gf}
All the experiments presented in this section are done on a quad-core 1.8
GHz 64-bit Intel Core I7 machine with 16 GBs memory. The software is
written in Matlab, with the linear solvers AGMG and LAMG working
via MEX interface.  All simulations were performed on a single core.
We summarize in~\cref{sec:parameters} a discussion on the values of the
different parameters that control our algorithm.

\subsection{Gradient Descent - Accelerated Gradient Descent - Gradient Flow}
In this section we compare in terms of accuracy and computational
efficiency the three time discretization schemes for~\cref{eq:gf-ode}
described in~\cref{sec:time-discretization}, Gradient Descent (GD),
Accelerated Gradient Descent (AGD), and Gradient Flow (GF).
\begin{figure}
  \centerline{
    \includegraphics[width=1\textwidth,
    trim={0.0cm 0cm 0.0cm 0.0cm},clip
    ]{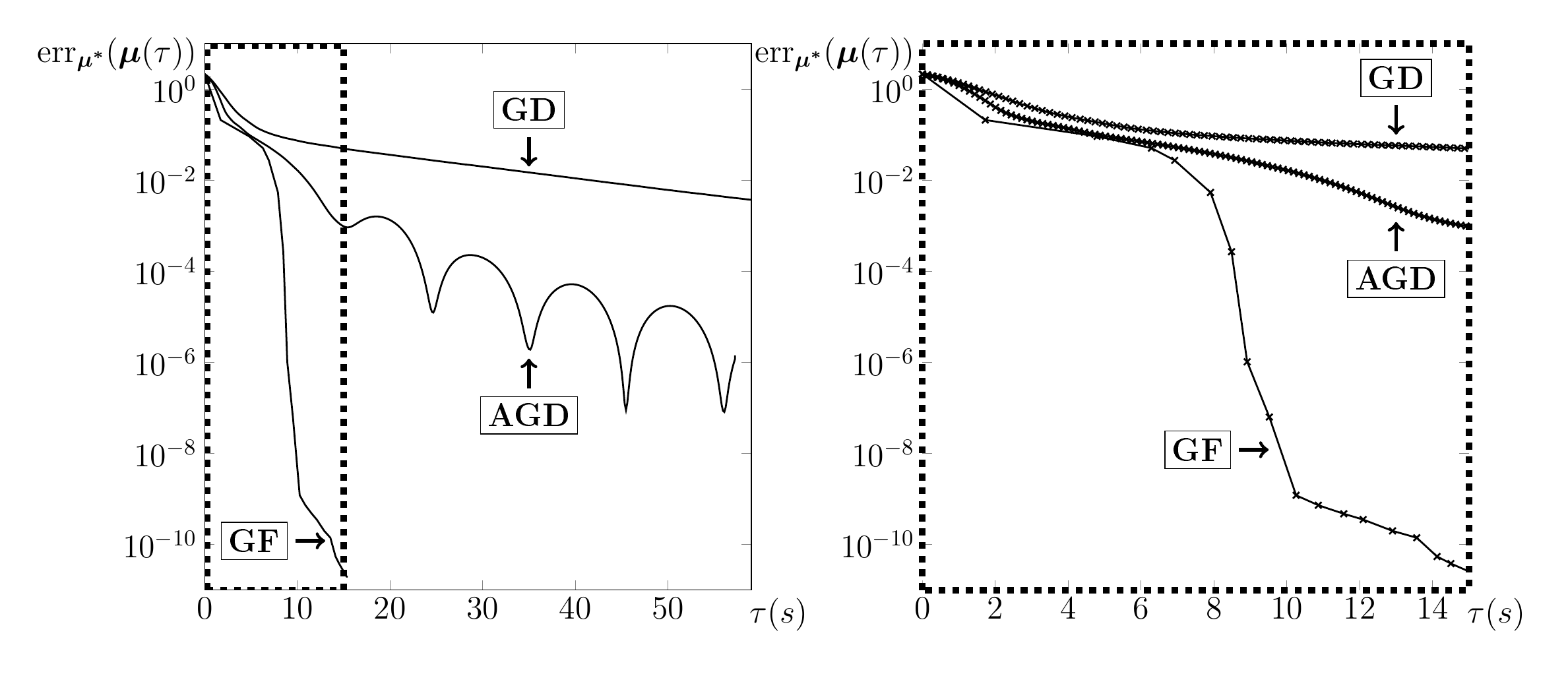}
  }
  \caption{ Plot of the quantity $\ErrTdens$ defined in
    \cref{eq:err-tdens} versus CPU time for the Gradient Descent (GD),
    Accelerated Gradient Descent (AGD) and the Gradient Flow (GF)
    approaches.  We report the results for Test-Case~2 using
      $\Graph_3$ (66049 nodes and 197120 edges).  The left panel
    reports the error evolution for the whole simulation.  The right
    panel, with dashed border, magnifies the portion of the left
    panel where the GF approach converged
    to the optimal solution.  }
  \label{fig:gd-agd-gf}
\end{figure}
To this aim, we consider only Test~case-2, where we have an explicit
formula for the optimal solution $\Vect{\Opt{\Tdens}}$
of~\cref{eq:min-lyap}, and we report in~\cref{fig:gd-agd-gf} the
evolution of the $\Vect{\Tdens}$-error
\begin{equation}
  \label{eq:err-tdens}
  \ErrTdens(\Vect{\Tdens})=
  \Norm[\Vect{\Length}]{\Vect{\Tdens}-\Vect{\Opt{\Tdens}}}
  /
  \Norm[\Vect{\Length}]{\Vect{\Opt{\Tdens}}},
\end{equation}
with respect to CPU time for the three approaches.  We present only
the results for $\Graph_3$ (66049 nodes and 197120 edges) since they
are representative of what we observed using smaller and larger
graphs.  The initial data used is
$\Vect{\Gfvar}^0=\Trans^{-1}(\Vect{1})$.  In the GD and the AGD
approaches, we start from an initial time step size
$\Deltat[\tstep]=0.1$ progressively increased by a factor $1.05$ until
a maximum of $\Deltat[\tstep]=2$, which was found experimentally to
guarantee stability of the schemes.  In the GF scheme we adopt the
Reduced-MG approach for the solution of the linear systems arising
from the Newton method, thus the time step size is tuned according to
condition~\cref{eq:tol-C}. For these experiments, the AGMG software
turns out to be the most efficient among the multigrid solvers for the
arising linear systems.

The left panel in~\cref{fig:gd-agd-gf} shows that the GF approach
outperforms the other two methods in terms of overall efficiency.  On
the right panel we magnify the portion of the left panel where the GF
approach converges toward the optimal solution.  Here, we can see that
each time step of GD and AGD requires much less computational effort
with respect to the GF approach, but more time steps are required.
The error-versus-time plot using the GD approach is monotonically
decreasing, but the convergence is extremely slow.  Using the ADG
scheme, the convergence toward the optimum is faster although
oscillations occur as the optimum is approached. These are both
well-known phenomena using this method (see, for
example,~\cite{su2014differential}).  During the initial steps of the
GF approach (from the first to the fourth) the error reduction per CPU
time is practically the same for all approaches. But then, the
convergence of the GF becomes much faster since we are getting closer
to the steady state, thus larger time step sizes $\Deltat[\tstep]$ are
used. In the intermediate steps ($5^{th}-9^{th}$), the GF approach
reduces the $\Vect{\Tdens}$-error by at least one order of magnitude
per step. In the final steps ($9^{th}-17^{th}$), the convergence
toward the optimal solution slows down due to some issues with the
linear solver, discussed in~\cref{sec:test-case-2}. When these
problems arise, the AGD approach may be a valid alternative to the GF
approach that guarantees a good trade-off between accuracy and
efficiency.

\subsection{Reduced-MG vs Full-LSC}
\label{sec:reduced-full}
In this section we compare the different approaches presented
in~\cref{sec-precs} for the solution of the saddle point linear
system. We take the graphs $\Graph_0,\Graph_1$ and $\Graph_2$ of
Test-case~1. We consider $\Vect{\Gfvar}^0=\Trans^{-1}(\Vect{1})$, and
we fix $\TolOpt=10^{-6}$.
\begin{table}
  \centering
  \begin {tabular}{cc|c|c|c|c|r<{\pgfplotstableresetcolortbloverhangright }@{}l<{\pgfplotstableresetcolortbloverhangleft }c}%
\toprule \multicolumn {2}{c|}{\textbf {Graph}} &\textbf {Approach} &\textbf {Time} &\multicolumn {1}{c|}{\textbf {Newton}} &\multicolumn {1}{c|}{\textbf {AGMG}} &\multicolumn {3}{c}{\textbf {CPU}} \\\textbf {$|\Nodes |$}&\textbf {$|\Edges |$}&\textbf { }&\textbf {steps}&\textbf {steps}&\textbf {Iters.}&\multicolumn {2}{c}{\textbf {(s)}}&\textbf {\%(L,B,W)}\\\midrule %
\pgfutilensuremath {1{,}089}&\pgfutilensuremath {3{,}136}&Full-LSC&\pgfutilensuremath {4}&\pgfutilensuremath {20}&\pgfutilensuremath {568}&$0$&$.7$&(83.3-16.7-00.0)\\%
&&Reduced-MG&\pgfutilensuremath {8}&\pgfutilensuremath {36}&\pgfutilensuremath {423}&$0$&$.14$&(83.6-12.2-04.2)\\\bottomrule %
\pgfutilensuremath {4{,}225}&\pgfutilensuremath {12{,}416}&Full-LSC&\pgfutilensuremath {4}&\pgfutilensuremath {27}&\pgfutilensuremath {1{,}092}&$4$&$.58$&(87.4-12.6-00.0)\\%
&&Reduced-MG&\pgfutilensuremath {8}&\pgfutilensuremath {40}&\pgfutilensuremath {589}&$0$&$.41$&(88.5-06.6-05.0)\\\bottomrule %
\pgfutilensuremath {16{,}641}&\pgfutilensuremath {49{,}408}&Full-LSC&\pgfutilensuremath {4}&\pgfutilensuremath {37}&\pgfutilensuremath {2{,}148}&$36$&$.5$&(89.1-10.9-00.0)\\%
&&Reduced-MG&\pgfutilensuremath {9}&\pgfutilensuremath {52}&\pgfutilensuremath {868}&$2$&$.07$&(92.4-04.8-02.8)\\\bottomrule %
\end {tabular}%

  \caption{ Comparison between preconditioning approaches.  The first
    two columns describe the dimensions of the graph considered. The
    remaining columns, from left to right, report the followings
    results: the total number of time steps required to achieve
    convergence $\TolOpt\leq 10^{-6}$, the total number of Newton
    steps (which is equal to the number of saddle point linear system
    solved), the total number of multigrid iterations performed (using
    AGMG).  The last column reports the CPU-time in seconds required
    to achieve convergence, and the percentage of CPU-time spent in
    solving linear systems (L), building the matrices (B), and wasted
    due to failures of linear and non-linear solvers (W).
  }
  \label{tab:reduced-full}
\end{table}
In this experiment AGMG turns out to be, again, the most efficient
multigrid solver for the linear systems with weighted Laplacian
matrices we need to solve in both approaches. The Full-LSC approach
requires less time steps, since the restriction in~\cref{eq:tol-detsV}
is less strict than that in~\cref{eq:tol-C}.  This results in fewer
linear systems to be solved and in fewer failures of the linear and
non-linear solvers, but the overall performance of the preconditioner
$\Matr{\PrecV}$ is worse. In fact, not only it requires the
approximate solution of two weighted Laplacian systems via multigrid
solver for each flexible GMRES iteration, but the total number of
iterations increases rapidly with the size of the graph.

The Reduced-MG approach may require a greater number of time steps and
non-linear iterations, resulting in a greater number of linear systems
to be solved. However, these linear systems involve weighted Laplacian
matrices that can be efficiently handled via algebraic multigrid
methods. 

\subsection{Test-case~1}
\label{sec:test-case-1}
In this section we present the results obtained for Test-case~1,
taking the graphs $(\Graph_{\Igraph})_{\Igraph=0,\ldots,5}$. In this
test case, we adopt the GF scheme combined with the reduced-MG
approach, using the AGMG algorithm as algebraic multigrid solver,
since it turns out to be the most efficient in terms of CPU time.  In
all experiments we used $\Vect{\Gfvar}^0=\Trans^{-1}(\Vect{1})$ as
initial data. We use a tight tolerance $\TolOpt=10^{-14}$ to declare
the convergence toward the steady state.  For this test-case, we
measure the accuracy of the method using the relative error with
respect to the optimal solution $\Vect{\Opt{\Pot}}$ of~\cref{eq:dual}
given in~\cref{eq:eik-optpot}, i.e.,
\begin{equation}
  \label{eq:err-pot}
  \ErrPot(\Vect{\Pot}):=
  {
    \Norm[2]{\Vect{\Pot}-\Vect{\Opt{\Pot}}}
  }{
    /\Norm[2]{\Vect{\Opt{\Pot}}}
  }.
\end{equation}
   
\begin{table}
  \centering
  \begin {tabular}{rrr|c|c|c|r<{\pgfplotstableresetcolortbloverhangright }@{}l<{\pgfplotstableresetcolortbloverhangleft }|c}%
\toprule \multicolumn {3}{c|}{\textbf {Graphs}} & \multicolumn {1}{c|}{\textbf {Time}} &\multicolumn {1}{c|}{\textbf {Newton}} &\multicolumn {1}{c|}{\textbf {AGMG}} &\multicolumn {2}{c|}{\textbf {CPU}} &\multicolumn {1}{c}{\textbf {Errors}} \\\textbf {$\Graph $}&\textbf {$|\Nodes |$}&\textbf {$|\Edges |$}&\textbf {steps}&\textbf {steps}&\textbf {Iters.}&\multicolumn {2}{c}{\textbf {(s)}}&$\ErrPot (\Vect {\MyOpt {\Pot }})$\\\midrule %
$\Graph _0$&\pgfutilensuremath {1.1\cdot 10^{3}}&\pgfutilensuremath {3.1\cdot 10^{3}}&\pgfutilensuremath {9}&\pgfutilensuremath {29}&\pgfutilensuremath {335}&$0$&$.0963$&3.3e-15\\%
$\Graph _1$&\pgfutilensuremath {4.2\cdot 10^{3}}&\pgfutilensuremath {1.2\cdot 10^{4}}&\pgfutilensuremath {6}&\pgfutilensuremath {25}&\pgfutilensuremath {359}&$0$&$.226$&2.7e-13\\%
$\Graph _2$&\pgfutilensuremath {1.7\cdot 10^{4}}&\pgfutilensuremath {4.9\cdot 10^{4}}&\pgfutilensuremath {7}&\pgfutilensuremath {26}&\pgfutilensuremath {399}&$0$&$.902$&9.0e-14\\%
$\Graph _3$&\pgfutilensuremath {6.6\cdot 10^{4}}&\pgfutilensuremath {2.0\cdot 10^{5}}&\pgfutilensuremath {7}&\pgfutilensuremath {28}&\pgfutilensuremath {456}&$4$&$.43$&3.3e-15\\%
$\Graph _4$&\pgfutilensuremath {2.6\cdot 10^{5}}&\pgfutilensuremath {7.9\cdot 10^{5}}&\pgfutilensuremath {7}&\pgfutilensuremath {31}&\pgfutilensuremath {545}&$24$&$.2$&6.7e-16\\%
$\Graph _5$&\pgfutilensuremath {1.1\cdot 10^{6}}&\pgfutilensuremath {3.1\cdot 10^{6}}&\pgfutilensuremath {8}&\pgfutilensuremath {34}&\pgfutilensuremath {646}&$136$&$$&5.0e-14\\\bottomrule %
\end {tabular}%

  \caption{Results for Test-case~1. From left to right, we report: the
    graph considered with the number of nodes $|\Nodes|$ and edges
    $|\Edges|$; the number of time steps required to achieve steady
    state, the total number of Newton steps (equal to the
    number of linear systems to be solved), the total number of
    multigrid iterations performed, the CPU time in seconds,
    and the error in the steady state solution
    $\Vect{\MyOpt{\Pot}}$ with respect to the optimal potential
    $\Vect{\MyOpt{\Pot}}$ defined in~\cref{eq:err-pot}.}
  \label{tab:eik}
\end{table}

The results in~\cref{tab:eik} show that the number of time steps
required to reach steady state is practically constant with respect to
the size of graph. The same holds with the total number of Newton
steps, which corresponds to the number of linear systems we need to
solve.  We attribute this phenomenon to the fact that, for this test
case, the graph considered is ``grid-like'' and the kernel of matrix
$\Lapl[\Vect{\Opt{\Cond}}]$ consists only of the constant vectors.
The experiments in the next sections show that when one of these two
conditions fails, the total number of linear systems to be solved
grows with the size of the graphs. We observed a modest increase in
the average number of multigrid iterations per linear system, which
goes from $12$ for the smallest ($\approx 10^3$ nodes/edges) to $ 19$
for the largest graph ($\approx10^6$ nodes/edges). This results in a
CPU time that grows slightly more than linearly with respect to the
graph size.  While in term of CPU time the proposed method is still
not comparable with algorithms dedicated to the solution of the SSSP,
such as those presented in~\cite{Madduri2007}, we are pleasantly
surprised by its accuracy and efficiency.

\begin{figure}
  \centerline{
    \includegraphics[width=1\textwidth,
    trim={0.0cm 0cm 0.0cm 0.0cm},clip
    ]{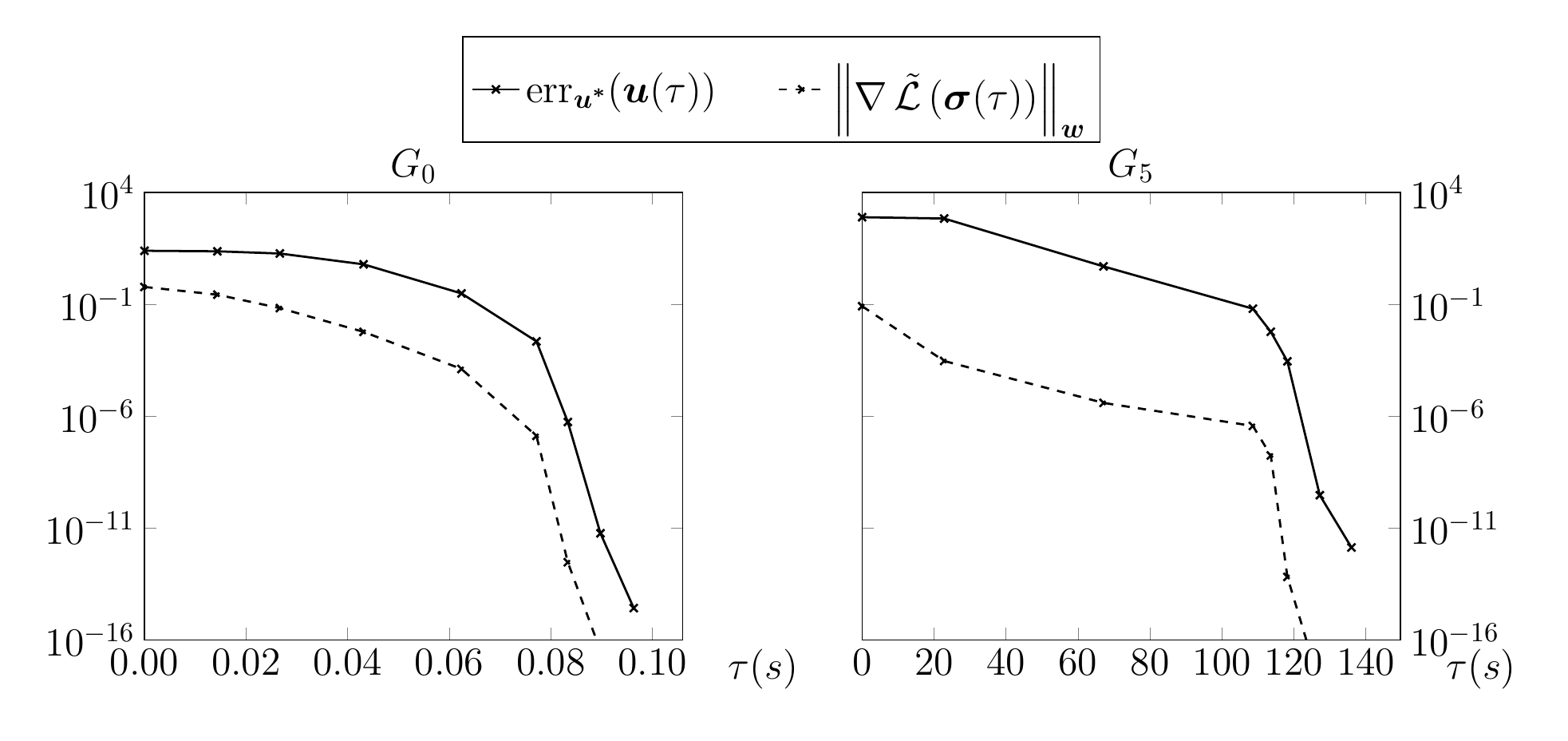}
  }
  \caption{ Plot of the $\Vect{\Opt{\Pot}}$-error (given in
    \cref{eq:err-pot}) and the norm of $\Grad
    \Ingf{\Lyap}(\Vect{\Gfvar})$ (used to estimate convergence toward
    steady state) against the CPU time for Test-Case-1, for $\Graph_0$
    (left) and $\Graph_5$ (right).  }
  \label{fig:eik-cpu-errors}
\end{figure}
In \cref{fig:eik-cpu-errors} we show the evolution of the relative
error with respect to the optimal potential $\Vect{\Opt{\Pot}}$ and
norm of the gradient of $\Ingf{\Lyap}(\Vect{\Gfvar})$ against the CPU
time, measured in seconds. Only the results for the smallest and
largest graphs ($\Graph_0$ and $\Graph_5$) are reported. The plots
show how most of the computational effort is spent during the initial
steps, and then, when
$\Norm[\Vect{\Length}]{\Grad\Ingf{\Lyap}(\Vect{\Gfvar})}$ reaches the
values $10^{-5}-10^{-7}$, the convergence accelerates since larger
time size steps are used, thus the time relaxation in~\cref{eq:gf-inf}
is so small that the GF approach becomes practically the pure Newton
method applied to~\cref{eq:bp}.
the 
\subsection{Test-case~2. Selection of the active sub-graph}
\label{sec:test-case-2}
The second test case turns out to be a more challenging problem for the
GF approach.  In fact, using the Reduced-MG approach, the efficiency
is strongly influenced by the evolution of the vector
\[
(\Mgrad \Vect{\Pot})^2-\Vect{1}
\; ,
\]
appearing in the definition of the matrix
$\Matr{\Cblock}=\Diag{\Vect{1}/\Deltat-((\Mgrad
  \Vect{\Pot})^2-\Vect{1})}$ in \cref{eq:jacobian-gf}. When
$\max((\Mgrad \Vect{\Pot})^2)<1$ or $\max((\Mgrad
\Vect{\Pot})^2)\approx 1$ a larger time step size $\Deltat[\tstep]$
and $\Damp=1$ (the un-damped Newton algorithm) can be used. On the
other hand, if $\max((\Mgrad \Vect{\Pot})^2)>1$, more time and Newton
steps are required to reach steady state. However, we note that most
of the edges where $\max((\Mgrad \Vect{\Pot})^2)>>1$ are those with
conductivity $\Vect{\Tdens}$ going to zero. Thus, the rows of the
matrix $\Lapl[\Vect{\Cond}]$ corresponding to the nodes that belong to
these edges became almost null.  Oscillation of the potential
$\Vect{\Pot}$ and thus in the gradient vector $\Mgrad \Vect{\Pot}$,
may appear when we get closer to the optimal solution
$\Vect{\Opt{\Gfvar}}=\Trans^{-1}(\Vect{\Opt{\Cond}})$, which contains
several zeros entries.

A simple strategy to cope with this problem is to divide the graph
$\Graph$ in two sub-graphs, $\Graph_{\On}=(\Nodes_{\On},\Edges_{\On})$
and $\Graph_{\Off}=(\Nodes_{\Off},\Edges_{\Off})$.  The first,
$\Graph_{\On}$, contains only those edges with conductivity
$\Vect{\Cond}$ above a fixed threshold $\TolCut>0$, and the nodes
connected by those edges. The second, $\Graph_{\Off}$, contains the
remaining edges and nodes. The time evolution is restricted to the
graph $\Graph_{\On}$. On the graph $\Graph_{\Off}$ the conductivity
$\Vect{\Cond}$ is set to zero on all edges in $\Edges_{\Off}$, while
the potential $\Vect{\Pot}$ restricted to node in $\Nodes_{\Off}$ is
fixed at the value given at the time of the selection. This
strategy closely resembles the Heuristic Optimality Check approach
considered in~\cite{Lorenz2015} for the solution of the Basis Pursuit
Problem. It may also be seen as an application of the Active Set
method in optimization~\cite{Zhang2012}.  On the other hand, from an
algebraic perspective, we are trying to identify and remove the
near-null vectors of matrix $\Lapl[\Vect{\Cond}]$.

\begin{table}
  \centering
  \begin {tabular}{r|ccc|rcc|cc}%
\toprule \multicolumn {1}{c|}{} &\multicolumn {3}{c|}{\textbf {Iterations}} &\multicolumn {3}{c@{\hspace {0.5em}}|}{\textbf {CPU}} &\multicolumn {2}{c}{\textbf {Errors}} \\\textbf {$\Graph $}&\textbf {T}&\textbf {N}&\textbf {AGMG}&\textbf {(s)}&\textbf {L,\!B,\!W\!(\%)}&\textbf {\!LP(s)\!}&$\left |\Norm [\infty ]{\Mgrad \Vect {\MyOpt {\Pot }}}-\Vect {1}\right |$&$\ErrTdens (\Vect {\MyOpt {\Cond }})$\\\midrule %
$\Graph _0$&\pgfutilensuremath {10}&\pgfutilensuremath {29}&\pgfutilensuremath {415}&\pgfutilensuremath {0.11}&(80-09-11)&\pgfutilensuremath {0.06}&2.3e-13&3.5e-11\\%
&\pgfutilensuremath {10}&\pgfutilensuremath {31}&\pgfutilensuremath {361}&\pgfutilensuremath {0.12}&(74-10-16)&&4.0e-14&8.4e-12\\\bottomrule %
$\Graph _1$&\pgfutilensuremath {12}&\pgfutilensuremath {36}&\pgfutilensuremath {539}&\pgfutilensuremath {0.39}&(83-06-11)&\pgfutilensuremath {0.54}&1.9e-10&1.8e-11\\%
&\pgfutilensuremath {10}&\pgfutilensuremath {38}&\pgfutilensuremath {523}&\pgfutilensuremath {0.39}&(80-08-12)&&1.0e-10&4.8e-13\\\bottomrule %
$\Graph _2$&\pgfutilensuremath {11}&\pgfutilensuremath {45}&\pgfutilensuremath {825}&\pgfutilensuremath {1.79}&(90-04-06)&\pgfutilensuremath {8.77}&2.1e-13&4.6e-12\\%
&\pgfutilensuremath {11}&\pgfutilensuremath {56}&\pgfutilensuremath {961}&\pgfutilensuremath {2.78}&(87-07-06)&&1.3e-11&2.5e-11\\\bottomrule %
$\Graph _3$&\pgfutilensuremath {11}&\pgfutilensuremath {56}&\pgfutilensuremath {1{,}166}&\pgfutilensuremath {14.4}&(93-03-04)&\pgfutilensuremath {189.74}&1.4e-10&2.0e-11\\%
&\pgfutilensuremath {11}&\pgfutilensuremath {65}&\pgfutilensuremath {1{,}323}&\pgfutilensuremath {15.3}&(89-06-05)&&1.3e-16&1.9e-12\\\bottomrule %
$\Graph _4$&\pgfutilensuremath {22}&\pgfutilensuremath {121}&\pgfutilensuremath {2{,}530}&\pgfutilensuremath {144}&(75-03-22)&\pgfutilensuremath {6{,}287}&3.3e-11&1.8e-11\\%
&\pgfutilensuremath {21}&\pgfutilensuremath {131}&\pgfutilensuremath {2{,}568}&\pgfutilensuremath {141}&(71-05-24)&&2.3e-16&9.2e-13\\\bottomrule %
$\Graph _5$&\pgfutilensuremath {\textbf{101}}&\pgfutilensuremath {329}&\pgfutilensuremath {5{,}312}&\pgfutilensuremath {1{,}230}&(83-04-13)&$\dagger$&$\dagger$&$\dagger$\\%
&\pgfutilensuremath {37}&\pgfutilensuremath {242}&\pgfutilensuremath {4{,}765}&\pgfutilensuremath {888}&(76-06-18)&&1.0e-15&8.9e-14\\\bottomrule %
\end {tabular}%

  \caption{ Numerical Results for Test-case~2. Each row corresponds to
    a different graph. Number of nodes/edges approximately scales by a
    factor of 4 between different graphs (exact numbers are reported
    in~\cref{tab:eik}). For each graph, the first row reports the data
    where the edge selection is not active, the second when the
    selection is active.  Starting from the second column we report:
    the number of time steps required to reach steady state with
    $\TolOpt=10^{-14}$ (\textbf{T}); the total number of Newton steps,
    which coincides with the linear system solved (\textbf{N}); the
    total number of AGMG iterations (\textbf{AGMG}); the CPU time in
    seconds, with the percentages of time spent in solving linear
    systems (\textbf{L}), building the matrices (\textbf{B}), and
    wasted by linear and non-linear solvers (\textbf{W}); the CPU time
    in seconds required to solve the problem via the Matlab linear
    programming solver (\textbf{LP}). The last columns reports the
    error of constraints of the dual problem and the relative
    $\Vect{\Tdens}$-error.  Convergence is not achieved for graph
    $\Graph_5$ when no selection is active (case marked with
    $\dagger$). In this case the Matlab Linear Programming solver did
    not find a solution within the 8 hours prescribed limit (again
    denoted by $\dagger$).}
  \label{tab:rect}
\end{table}
We summarize in~\cref{tab:rect} the numerical results for Test-case~2,
with and without the edge selection.  While for the smallest graphs
the differences are relatively small, for the largest one, not using
the selection procedure leads to failure of the convergence of the
scheme, because edges outside the ``support'' of the optimal solution
$\Vect{\Opt{\Cond}}$, start having $\ABS{(\Mgrad \Vect{\Pot})^2}>>1$.
By the restriction imposed in~\cref{eq:tol-C}, this forces us to
shrink the time step size $\Deltat[\tstep]$ below the value $2$, where
we can use the GD and the AGD approaches at a lower computational
cost. As already mentioned in~\cref{sec:gd-agd-gf}, the GD and the AGD
approaches can represent robust alternatives to the GF approach, since
in both cases the gradient direction is given by the product of the
$\Vect{\Gfvar}^{\tstep}$ and $\ABS{(\Mgrad
  \Vect{\Pot}^{\tstep})}^2-\Vect{1}$, thus on the edges where
$\Vect{\Gfvar}^{\tstep}$ tends to zero, the first factor annihilates
oscillation and errors in the second.  However, the selection
procedure removes the edges with strong gradient oscillations, thus a
larger time step size can be used, hence convergence toward the
approximate steady state is achieved in fewer time steps.  Moreover,
it reduces the overall computational cost since $\Graph_{\On}$
contains fewer edges and nodes. The drawback of this approach is that
it requires to tune the parameter $\TolCut$.  We found experimentally
that $\TolCut=10^{-9}$ is a suitable threshold for the test-cases
considered.  In all simulations we obtain highly accurate
approximations of the optimal solutions and the constraints on the
vector $\Mgrad \Vect{\Pot}$.

\begin{figure}
  \centerline{
    \includegraphics[width=1\textwidth,
    trim={0.0cm 0cm 0.0cm 0.0cm},clip
    ]{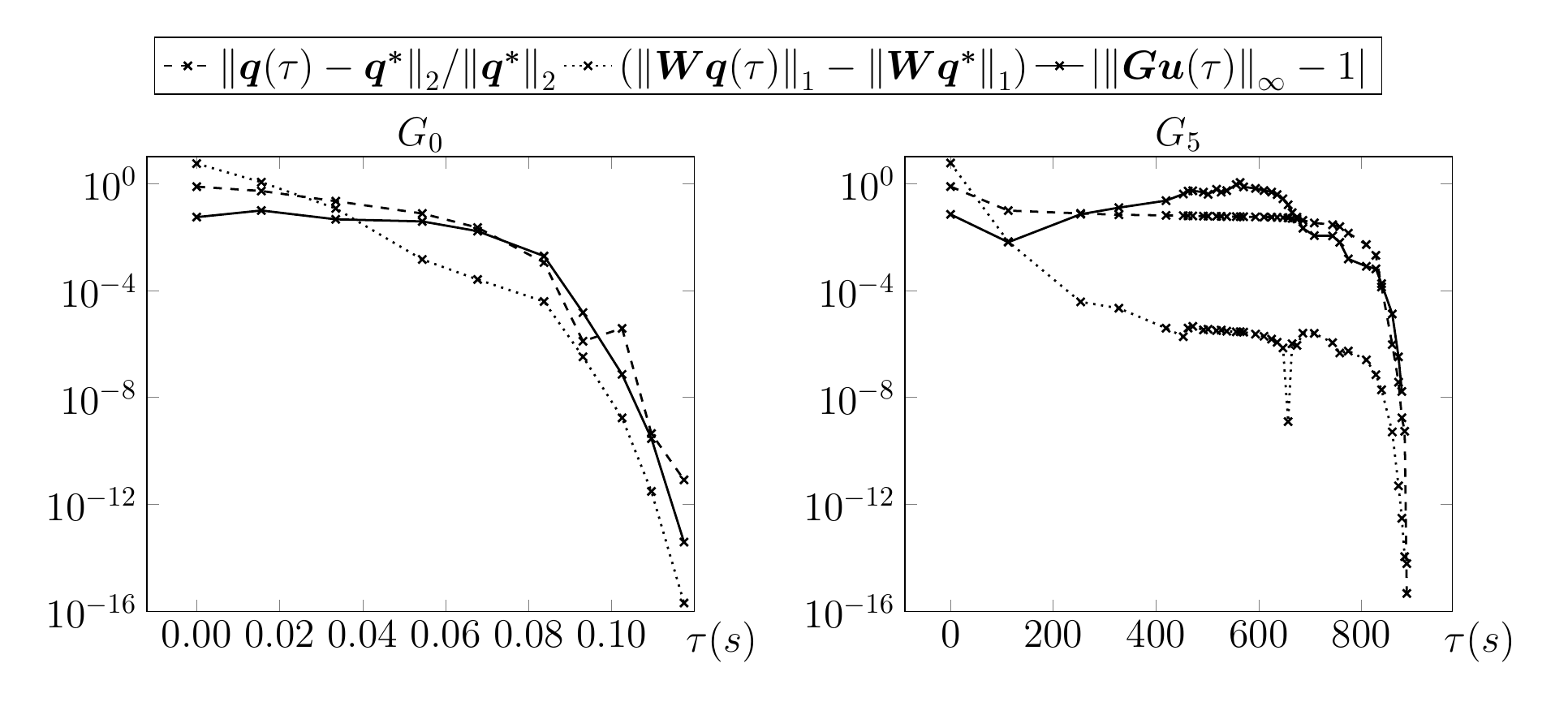}
  }
  \caption{ Plot of the relative error with respect to the optimal
    flux $\Vect{\Opt{\Dflux}}$, the error in the value of the
    minimization problem, and the error of the constraints
    of~\cref{eq:dual}. We report the data against the CPU time $\tau$,
    only for Test-Case-2 with $\Graph_0$ (left) and $\Graph_5$
    (right).  }
  \label{fig:rect-cpu-errors}
\end{figure}
In order to show the accuracy and the efficiency of the GF approach
along the whole optimization process, we present
in~\cref{fig:rect-cpu-errors} the evolution of three different errors
against the CPU time.  In particular, we report the relative error
with respect to the optimal solution $\Vect{\Opt{\Dflux}}$, the error
of the cost functional in~\cref{eq:beckmann}, which estimates the
error in computation of the Wasserstein distance, and the error of the
constraint of~\cref{eq:dual}. We report only the results for the
smallest and the largest graph, $\Graph_0$ and $\Graph_5$, using the
Reduced-MG approach with selection of the active edges with
$\TolCut=10^{-9}$. \Cref{fig:rect-cpu-errors} shows how having a good
approximation of the value of \cref{eq:beckmann} does not guarantee
having a good solution of the problem, since ~\cref{eq:beckmann} is
convex but not strictly.  Similar to Test-case~1, most of the
computational effort is spent during the initial phase where the
errors are reduced approximately to $10^{-4}$. Then the convergence
accelerates and the optimal solutions of~\cref{eq:beckmann,eq:dual}
are found in few time steps.  On the left panel
of~\cref{fig:rect-cpu-errors} , which shows the results for
$\Graph_0$, we can see that the $\ell^{\infty}$-norm of $\Mgrad
\Vect{\Pot}(\tau)$ converges rapidly to $1$, thus large time step
sizes are used, and thus few time steps are required to reach steady
state.  On the right panel, which shows the results for $\Graph_5$, we
can see that when the $\ell^{\infty}$-norm of the gradient vector
$\Mgrad \Vect{\Pot}(\tau)$ exceeds 1 we are forced to shrink the time
step size, resulting in more but less costly time steps.

\subsection{Test-case~3}
In this section we summarize the numerical results obtained for the
graphs of the Test-case~3.  In this case, the LAMG solver turns out to
be more efficient and robust than the AGMG solver, which does not
converge within the maximum number of multigrid iterations (fixed
at 100) when we consider the largest graphs.  In all simulations, we
started from $\Vect{\Gfvar}^0=\Trans^{-1}(\Vect{1})$, with the edge selection
strategy active.
\begin{table}
  \begin {tabular}{c|cc|ccc|r<{\pgfplotstableresetcolortbloverhangright }@{}l<{\pgfplotstableresetcolortbloverhangleft }|rrrr}%
\toprule \multicolumn {1}{c|}{\textbf {TC}} & \multicolumn {2}{c|}{\textbf {Graph}} & \multicolumn {3}{c|}{\textbf {Iterations}} & \multicolumn {2}{c|}{\textbf {CPU}} & \multicolumn {2}{c}{\textbf {$\left |\Norm [\infty ]{\Mgrad \Vect { \MyOpt {\Pot }}}-\Vect {1}\right |$}} & \multicolumn {2}{c}{\textbf {$\Norm [2]{\Matr {\Diff }\Vect {\MyOpt {\Dflux }}-\Vect {\Forcing }}/\Norm [2]{\Vect {\Forcing }}$}} \\&\textbf {$|\Nodes |$}&\textbf {$|\Edges |$}&\textbf {T}&\textbf {N}&\textbf {LAMG}&\multicolumn {2}{c|}{\textbf {(s)}}&min&max&min&max\\\midrule %
\textbf {WS}&1e3&2e3&\pgfutilensuremath {11}&\pgfutilensuremath {58}&\pgfutilensuremath {494}&$3$&$.6$&1.4e-13&4.4e-9&3.3e-11&1.2e-9\\%
10\%&1e4&2e4&\pgfutilensuremath {23}&\pgfutilensuremath {147}&\pgfutilensuremath {1{,}665}&$24$&$.1$&1.7e-12&1.4e-8&2.6e-11&3.9e-9\\%
&1e5&2e5&\pgfutilensuremath {46}&\pgfutilensuremath {324}&\pgfutilensuremath {4{,}791}&$387$&$.6$&3.7e-11&7.6e-7&1.7e-14&2.9e-9\\%
&1e6&2e6&\pgfutilensuremath {86}&\pgfutilensuremath {670}&\pgfutilensuremath {12{,}717}&$8{,}844$&$.4$&1.8e-10&1.7e-6&1.0e-10&7.4e-9\\%
\bottomrule \textbf {WS}&1e3&2e3&\pgfutilensuremath {14}&\pgfutilensuremath {70}&\pgfutilensuremath {537}&$4$&$$&6.2e-12&2.7e-8&4.2e-12&8.4e-9\\%
100\%&1e4&2e4&\pgfutilensuremath {34}&\pgfutilensuremath {202}&\pgfutilensuremath {1{,}979}&$32$&$.1$&6.2e-11&5.8e-8&6.6e-12&1.3e-9\\%
&1e5&2e5&\pgfutilensuremath {73}&\pgfutilensuremath {441}&\pgfutilensuremath {5{,}602}&$490$&$.4$&2.5e-12&6.7e-8&3.1e-11&8.8e-10\\%
&1e6&2e6&\pgfutilensuremath {140}&\pgfutilensuremath {882}&\pgfutilensuremath {13{,}362}&$12{,}910$&$$&4.3e-9&2.2e-6&3.9e-10&3.3e-9\\%
\bottomrule \textbf {ER}&1e2&1e3&\pgfutilensuremath {8}&\pgfutilensuremath {35}&\pgfutilensuremath {71}&$6$&$$&5.6e-15&2.5e-7&9.3e-16&8.2e-9\\%
10\%&1e3&1e4&\pgfutilensuremath {17}&\pgfutilensuremath {99}&\pgfutilensuremath {857}&$9$&$.1$&9.5e-12&3.7e-8&5.0e-13&7.5e-10\\%
&1e4&1e5&\pgfutilensuremath {41}&\pgfutilensuremath {240}&\pgfutilensuremath {2{,}680}&$101$&$$&3e-11&5.1e-7&2.3e-11&9.1e-9\\%
&1e5&1e6&\pgfutilensuremath {51}&\pgfutilensuremath {325}&\pgfutilensuremath {4{,}559}&$1{,}783$&$$&7.3e-10&1.3e-6&1.5e-12&5.4e-9\\%
\bottomrule \textbf {ER}&1e2&1e3&\pgfutilensuremath {9}&\pgfutilensuremath {43}&\pgfutilensuremath {85}&$2$&$.4$&5.1e-13&4.4e-9&4.0e-11&3.3e-9\\%
100\%&1e3&1e4&\pgfutilensuremath {19}&\pgfutilensuremath {107}&\pgfutilensuremath {900}&$8$&$.5$&1.3e-10&3.8e-8&1.4e-11&4.9e-9\\%
&1e4&1e5&\pgfutilensuremath {37}&\pgfutilensuremath {227}&\pgfutilensuremath {2{,}631}&$102$&$.2$&1.5e-11&6.5e-8&2.4e-12&5.3e-9\\%
&1e5&1e6&\pgfutilensuremath {81}&\pgfutilensuremath {498}&\pgfutilensuremath {7{,}319}&$2{,}215$&$$&6.4e-10&1.7e-5&5.5e-12&6.7e-9\\%
\bottomrule \textbf {BA}&1e3&4e3&\pgfutilensuremath {13}&\pgfutilensuremath {69}&\pgfutilensuremath {501}&$5$&$$&1.9e-14&3.6e-8&8.0e-12&8.0e-9\\%
10\%&1e4&4e4&\pgfutilensuremath {26}&\pgfutilensuremath {149}&\pgfutilensuremath {1{,}453}&$45$&$.5$&1.2e-13&3.3e-9&7.4e-14&4.0e-9\\%
&1e5&4e5&\pgfutilensuremath {58}&\pgfutilensuremath {371}&\pgfutilensuremath {4{,}868}&$769$&$$&2.3e-11&8.3e-8&1.3e-11&4.3e-9\\%
&1e6&4e6&\pgfutilensuremath {108}&\pgfutilensuremath {724}&\pgfutilensuremath {11{,}638}&$20{,}920$&$$&2.1e-10&5.2e-7&3.2e-11&2.6e-9\\%
\bottomrule \textbf {BA}&1e3&4e3&\pgfutilensuremath {18}&\pgfutilensuremath {98}&\pgfutilensuremath {717}&$6$&$.5$&3.8e-15&6.3e-9&1.5e-14&6.4e-9\\%
100\%&1e4&4e4&\pgfutilensuremath {36}&\pgfutilensuremath {214}&\pgfutilensuremath {2{,}082}&$56$&$.7$&6.7e-12&1.4e-8&4.5e-13&3.1e-9\\%
&1e5&4e5&\pgfutilensuremath {68}&\pgfutilensuremath {425}&\pgfutilensuremath {5{,}536}&$911$&$.8$&5.5e-9&2.4e-7&1.0e-11&7.8e-9\\%
&1e6&4e6&\pgfutilensuremath {133}&\pgfutilensuremath {860}&\pgfutilensuremath {13{,}914}&$30{,}160$&$$&1.1e-8&3e-7&4.6e-12&5.9e-9\\\bottomrule %
\end {tabular}%

  \caption{ Results for Test-case~3. The first column describes the
    test case (\textbf{TC}) considered, distinguishing among the
    Erd\"os-R\'enyi (\textbf{ER}), Watts-Strogatz (\textbf{WS}), and
    the Barabasi-Albert (\textbf{BA}) graphs. It also indicates the
    percentage of non-zero elements of $\Vect{\Forcing}$, $10\%$ or
    $100\%$. The second and third columns describe the dimensions of
    the graphs considered. From left to right, we report: the total
    number of time steps required to achieve convergence
    (\textbf{T}), the total number of Newton steps (\textbf{N}), the
    total number of LAMG iterations performed (\textbf{LAMG}), and the
    CPU-time in seconds required to achieve convergence.  The results
    reported are the average of 10 different graphs with similar
    dimensions.  The last columns reports the minimum and the maximum
    of the error in the constraint of \cref{eq:dual,eq:beckmann}. In
    all simulation the duality gap in~\cref{eq:duality-gap} is close
    to machine precision.  }
  \label{tab:ws}
\end{table}
We summarize in~\cref{tab:ws} the averaged results for the different
groups of problems of Test-case~3.  We observed an increased number of
time steps and Newton steps with the size of the graph, with
  the later scaling, approximately, between
$\mathcal{O}(\Nedge^{0.31})$ and $\mathcal{O}(\Nedge^{0.36})$. This
result compares well with the theoretical upper bound presented
in~\cite{Madry:2013}, which estimates the cost of
solving~\cref{eq:beckmann} on bipartite graphs with the solution of
$\mathcal{O}(\Nedge^{3/7\approx 0.428})$ Laplacian systems.
Unfortunately, as we already observed in Test case~1 and~2, the number
of multigrid iterations per linear system increases with the size of
the graphs. Nevertheless, the resulting total CPU time scales certainly
better than the $\mathcal{O}(\Nnode^2)$-time for solving the
$\Lspace{1}$-\OTP\ described in \cite{Ling-et-al:2007}.

\subsection{Parameter tuning}
\label{sec:parameters}
In this section we want to summarize the crucial role played by the
different parameters introduced above.  In particular the non-linear
tolerance $\TolNonLinear$, the linear solver precision
$\TolLinearNewton$, the bound $\TolC$ in~\cref{eq:tol-C}, and the
lower bound $\MinDamp$ for the damping parameter $\Damp$.  These
parameters play competing roles in the accuracy and the performance of
the algorithm. For example, the value of $\TolC$ influences the time
step size $\Deltat[\tstep]$. Typically, taking smaller time steps
results in a higher number of time steps, but less non-linear
iterations, each one requiring the solution of the linear system
in~\cref{eq:newton}. Smaller time steps size also result in fewer
failures of the linear and non-linear solvers. A second example is
given by the lower bound $\MinDamp$ imposed to the damping parameter
$\Damp$. Smaller values of $\MinDamp$ results in fewer failures of the
non-linear solver, but an increased number of damped Newton steps. In
certain cases, it may be more convenient to perform two time steps
with $\Deltat$ rather than one single time step with $2\Deltat$.

Finding a good combination among all these parameters, in order to
obtain fast and accurate solution of the \OTP\ on
graphs, is not a trivial task.  In~\cref{tab:parameters} we summarize
a combination of parameters that we found experimentally to give a
good trade-off between accuracy and efficiency in finding the solution
of the \OTP\ on graphs.
\begin{table}
  \begin{center}
    \begin{tabular}{l|r} 
      \textbf{Parameter} & \textbf{Value } \\
      \hline
      $\TolNonLinear$ (non-linear tolerance in~\cref{eq:non-linear}) & $10^{-8}$\\
      $\TolLinearNewton$ (linear solver tolerance in~\cref{eq:relative-res}) & $10^{-4}$ \\
      $\TolC$ ($\Deltat$ tuning in~\cref{eq:tol-C}) & $10^{-8}$ \\
      $\MinDamp$ (lower bound for $\Damp$) & $5\times10^{-2}$ \\
      $\TolOpt$ (convergence criterion in \cref{eq:steady-state}) & $10^{-12}$ \\
      $\TolCut$ (selection threshold ) & $10^{-9}$ \\
      $\Rmax$ (max. Newton steps ) & 30 \\
      \hline
    \end{tabular}
    \caption{Parameters used in most of the numerical experiments.}
    \label{tab:parameters}    
  \end{center}
\end{table}

\section{Conclusions}
In this paper we have presented a novel gradient descent
approach for the numerical solution of the $\Lspace{1}$-OTP on
graphs. We have also presented different numerical approaches based
on this gradient descent dynamics.
The GF method with the Reduced-MG
approach is found to be an accurate, efficient, and robust solver for
the $\Lspace{1}$-OTP on graphs.  Different types of graphs have been
used, with different forcing terms.  The numerical experiments show
that the number of time steps with the GF-Reduced-MG approach ranges
between being constant and scaling proportionally to
$\mathcal{O}(\Nedge^{0.36})$ depending on the type of graph $\Graph$
and the forcing term $\Vect{\Forcing}$. We observed the same scaling
in the total number of Newton steps, which coincides with the number
of linear systems with weighted Laplacian matrix we need to solve.
The selection of active edges/nodes described
in~\cref{sec:test-case-2} drastically affects the efficiency of the
GF-Reduced-MG approach, even if it requires tuning a parameter that is
problem-dependent.

The proposed approach is iterative, thus any approximate solution can
be used as initial data for the approximation process.  Thus it may be
integrated with a fast but inaccurate algorithm
for~\cref{eq:beckmann}, since most of the computation effort is spent
in the initial steps, while closer to the optimum the convergence of
the GF accelerates.  This can be accomplished by introducing a
multi-level version of the algorithm, like in~\cite{behata:2010},
where an initial solution is obtained interpolating an approximate
solution computed on a coarser graph. The coarsening/prolongation
strategy used by the AGMG and the LAMG softwares should provide an
excellent tool to this aim. Parallelization of the multigrid linear
solver is also a future improvement that we want to introduce in our
method, in order to tackle problems on large-scale graphs. Depending
on the graph structure, optimal scalability of the multigrid solver
(which accounts for almost all the computation effort) can be
achieved.

A second line of research should be aimed in making more
robust and less problem-dependent the algorithm. In particular, the
selection of the active-inactive portion of the graph described
in~\cref{sec:test-case-2} may be replaced by the identification, via
algebraic methods, of the near-null space of the weighted Laplacian
matrices, similar to the approach proposed
in~\cite{Bergamaschi-et-al:2019}. Last, a deeper theoretical
investigation of the proposed time-stepping scheme may help in
defining new and more robust strategies to tune the parameters,
summarized in~\cref{tab:parameters}, that govern our algorithm.

\bibliographystyle{unsrt}
\bibliography{Strings,biblio_abbr}

\end{document}